\documentclass[11pt,a4paper,reqno]{amsart}
\usepackage[includehead,includefoot,margin=25mm]{geometry}

\usepackage{enumerate}
\usepackage{amsmath}
\usepackage{amsthm}
\usepackage{amssymb}
\usepackage{color}
\usepackage[pagewise,mathlines]{lineno}
\usepackage{hyperref}
\usepackage{amsfonts}
\usepackage[english]{babel}
\usepackage[none]{hyphenat}

\newcommand{\Di}{\mathcal{D}}
\newcommand{\Ka}{\mathcal{K}}
\newcommand{\R}{\mathbb R}
\newcommand{\eps}{\varepsilon}
\newcommand{\expt}{\text{exp}}

\newcommand{\La}{\mathcal{L}}

\newtheorem{lemma}{Lemma}[section]
\newtheorem{cor}{Corollary}[section]
\newtheorem{prop}{Proposition}[section]
\newtheorem{thm}{Theorem}[section]

\newtheorem*{step1}{Step 1}
\newtheorem*{step2}{Step 2}

\newtheorem*{MPL}{Mountain Pass Theorem (Ambrosetti-Rabinowitz)}

\theoremstyle{proof}
\newtheorem*{demTeo}{Proof of Theorem \ref{TeoBestConst}}
\newtheorem*{SubcriticalProof1}{Proof of Theorem \ref{subcritical} for $p<q$}
\newtheorem*{SubcriticalProof2}{Proof of Theorem \ref{subcritical} for $q\leq p$}
\newtheorem*{ExistenceProof}{Proof of Theorem \ref{ExistsSolution}}
\newtheorem*{proofEquiv}{Proof of Theorem \ref{FractSpaces}}
\theoremstyle{remark}
\newtheorem{example}{Example}[section]

\begin{document}

\title{non-local equations and optimal  Sobolev inequalities on compact manifolds}

\author[C. Rey]{Carolina Ana Rey}
\address{Departamento de Matem\'aticas, Universidad T\'ecnica
	Fe\-de\-ri\-co San\-ta Ma\-r\'\i a, Valpara\'\i
	so, Chile}  \email{carolina.reyr@usm.cl}

\author[N. Saintier]{Nicolas Saintier}
\address{Departamento de Matem\'atica, Facultad de Ciencias Exactas y Naturales, Universidad de Buenos Aires,
Argentina}  \email{nsaintie@dm.uba.ar}

\thanks{
	\mbox{\hspace{11pt}}The first author was supported by Fondecyt Project 3200422.}

\begin{abstract} 
	This paper deals with fractional Sobolev spaces  on  a compact Riemannian manifold. 
    We prove a Sobolev inequality in the critical range with an optimal 
	constant for these fractional Sobolev spaces. We use this result to study the existence of 
	a non-trivial solution for equations driven by a non-local integro-differential operator $\La_{\Ka}$ 
with critical non-linearity. 
\end{abstract}

\maketitle

\section{Introduction}
 
Equations involving non-local operators on an open subset of $\R^n$ are the subject of an intense research activity which focuses on the classical questions of existence, uniqueness, regularity, and qualitative properties of 
the solutions, for both linear and non-linear operators, like the fractional  and the p-fractional Laplacian (see \cite{caffa, KKP, servadei2013yamabe, silvestre}). This acute effort to understand this type of equation is due to its multiple applications in several contexts: continuum mechanics, phase transition, population dynamics, optimal control, game theory and image processing, as is explained in \cite{caffarelli2012non,gilboa2009nonlocal} and references therein.

 In this paper, we are particularly interested in equations on Riemannian manifolds involving a non-linearity whose growth is critical from the point of view of the Sobolev embedding. The first results in this direction started in the '60s in the context of the Yamabe problem, a classical problem in differential geometry formulated as the question of finding a non-trivial solution to a particular critical equation with the Laplace-Beltrami operator. The problem was solved completely three decades ago, and since then,  many authors have devoted their work to extend the techniques used for its resolution in different directions  (e.g. \cite{brezis1983positive,djadli2000paneitz}). 
 Recently an analogous theory has been developed for the so-called fractional Yamabe problem, 
which relies on finding a solution to a critical equation with a specific fractional operator on manifolds (\cite{chang2011fractional,gonzalez2013fractional,kim2017existence}). 

To the best of our knowledge, there is still no work dealing with critical equations on manifolds for general linear and non-linear integro-differential operators. 
The present work aims to provide a functional framework suitable to extend to the p-fractional Laplacian on a compact Riemannian manifold $(M,g)$ many results obtained for the standard Laplacian.
More precisely, we start defining natural fractional  Sobolev spaces $W^{s,p}(M)$, and then we prove an optimal Sobolev inequality. Moreover, we deduce from there an existence result for critical equations involving a singular non-local operator.

The question of  suitable natural fractional Sobolev spaces on manifolds is not trivial. 
Indeed we found two different definitions of fractional Sobolev spaces in the literature, 
both equally natural and appealing but slightly different.
H. Triebel introduced in \cite{T2, T} the whole scales of Besov $B^s_{p,q}(M)$ and Triebel–Lizorkin spaces $F^s_{p,q}(M)$ on manifolds. 
The Triebel-Lizorkin spaces on manifolds are modelled on the classical Triebel-Lizorkin spaces on $\R^n$ 
through exponential charts and partition of unity. On the other hand, the Besov spaces on manifolds are defined by interpolation of the Triebel-Lizorki spaces, as in $\R^n$.
Both spaces have attracted considerable attention in the last decades. 
For example in \cite{TPV}, the authors study the theory of Besov and 
Triebel–Lizorkin spaces on general non-compact Lie groups endowed with a sub-Riemannian structure.
Furthermore, it was recently proved in \cite{BBM, BBM2002} that there exists a Brezis-Bourgain-Mironescu type result for these spaces: 
$$ \lim_{s\to 1} (1-s)[u]_{s,p}^p= C \|\nabla u\|_p^p. $$ 

In the Euclidean setting, the fractional Sobolev spaces are defined for $s\in (0,1),\,sp<n$, as 
\begin{equation}\label{DefFracSobRn} 
W^{s,p}(\R^n) := \{ u\in L^p(\R^n),\, [u]^p_{s,p} <\infty\},
\end{equation}  
endowed with the norm $\|u\|_{s,p}^p=\|u\|_p^p+[u]_{s,p}^p$, where 
\begin{equation}\label{DefSemiNormRn}
[u]_{s,p}^p:= \iint_{\R^n\times \R^n} \frac{|u(x)-u(y)|^p}{|x-y|^{n+sp}}\,dxdy,
\end{equation} 
 known as the Galgliardo seminorm.
Basic properties of $W^{s,p}(\R^n)$ can be found in \cite{NPV}, 
whereas the complete theory of Besov and Triebel-Lizorkin spaces on $\R^n$ can be found in \cite{triebeltheoryof}. In particular, it holds 
$$
F^s_{p,p}(\R^n)=B^s_{p,p}(\R^n)=W^{s,p}(\R^n).
$$
H. Triebel extended part of this theory to $B^s_{p,q}(M)$ and $F^s_{p,q}(M)$ in \cite{T2, T}. 
Furthermore, he proved that when $p=q$, we have $F^s_{p,p}(M)=B^s_{p,p}(M)$, which are the spaces we are interested in, and we call $W^{s,p}(M)$. 
More recently, L. Guo, B. Zhang, and Y. Zhang \cite{GZZ} straightforwardly adapted \eqref{DefFracSobRn}-\eqref{DefSemiNormRn} on a compact Riemannian manifold 
$(M,g)$ in the following way. For $s\in (0,1),\,sp<n$, they define
\begin{equation}\label{DefFracSob} 
	\widetilde{W}^{s,p}(M) := \{ u\in L^p(\R^n),\, [u]^p_{s,p} <\infty\},
\end{equation}  
 where 
\begin{equation}\label{DefSemiNorm} 
	[u]_{s,p}^p:= \iint_{M\times M} \frac{|u(x)-u(y)|^p}{d_g(x,y)^{n+sp}}\,dv_g(x)dv_g(y).
\end{equation} 
 
Both fractional spaces $W^{s,p}(M)$ and $\widetilde{W}^{s,p}(M)$ are pretty natural and satisfy the usual properties of Sobolev spaces 
(such as Banach, reflexivity, the density of smooth functions and embedding Theorems). 
So, it is logical to investigate if they define different spaces or not. 
Our first main result shows that they indeed coincide: 

\begin{thm}\label{FractSpaces} Let $(M,g)$ be a compact Riemannian manifold without boundary of dimension $n$, $s\in (0,1)$ and $sp<n$. 
Then the fractional Sobolev spaces $W^{s,p}(M)$  and $\widetilde{W}^{s,p}(M)$ introduced by H. Triebel and  
L. Guo, B. Zhang, and Y. Zhang coincide with norm equivalence. 
\end{thm}

From now on, we fix a compact Riemannian manifold boundaryless of dimension $n$ and denote
$W^{s,p}(M)=\widetilde{W}^{s,p}(M)$. 
As a consequence of H. Triebel and L. Guo, B. Zhang, and Y. Zhang's work, we know in particular that 
the embedding  $W^{s,p}(M)\hookrightarrow L^q(M)$ is continuous for $q\le p^*$ and compact for $q < p^*$, 
 where $p^*$ is the fractional critical Sobolev exponent given by 
 \begin{equation} 
 p^* = p^*(n,s)=\frac{np}{n-sp}.   
 \end{equation} 

The p-fractional Laplacian $(-\Delta_p)^s:W^{s,p}(M)\to (W^{s,p}(M))'$ appears naturally when looking for critical points of the semi-norm $[u]_{s,p}^p$. So we consider the equation
\begin{equation}\label{MainEqu100}
(-\Delta_p)^su + h|u|^{p-2}u = f|u|^{q-2}u
\end{equation}
where $f$ and $h$ are smooth functions on $M$, with $q\le p^*$.  
Solutions to \eqref{MainEqu100} can be found as critical points of the functional 
$$ J(u) := \frac{1}{p} [u]_{s,p}^p + \frac{1}{p} \int_M h|u|^p\,dv_g - \frac{1}{q}\int_M f|u|^q\,dv_g, 
\qquad u\in W^{s,p}(M). $$ 
Actually, we shall consider a more general non-local operator $\La_{\Ka}$ on $W^{s,p}(M)$ defined as follows.
Denote $\Di=\{(x,x):x\in M\}$. Let $g_\eps$ be the metric on $\R^n$ obtained blowing-up $g$ at $x_0$, namely 
$g_\eps(x) = ((\exp_{x_0})^*g)(\eps x)$.
Let $\Ka( \cdot, \cdot \ ;g):(M\times M)\backslash \Di\rightarrow(0,\ +\infty)$ be a function which satisfies:

\begin{enumerate}[(K1)]
	
	\item\label{K1} $m\Ka\in L^{1}(M\times M)$ , where $m=m(x,y)=\displaystyle \min\{d_g(x,y)^{p},1\};$ 
	
	\item \label{K2} $\Ka(x, y; g)=\Ka(y, x ;g)$ for any $(x,y)\in (M\times M)\backslash \Di$.
	
	\item \label{K3} There is a constant $\Lambda>1$ such that 
	$$
	\Lambda^{-1} < \Ka(x,y;g) d_g(x,y)^{n+ps}  < \Lambda \, \text{ for all } x,y \in (M\times M)\backslash \Di.
	$$
	\item \label{K4} Let $x_0\in M$ and $G:T_{x_0}M\to M$ be a smooth function. If we denote
	$$ 
	\tilde \Ka(X,Y,G^* g):=\Ka(G(X),G(Y);g)\qquad \text{for all } X \neq Y\in T_{x_0}M,
	$$ 
	then it holds
	$$
	\eps^{n+sp}\tilde \Ka(X,Y;\eps^2 g_\eps)\to |X-Y|^{-(n+ps)}
	$$ 
	as $\eps\to 0$ uniformly on compacts.
\end{enumerate}

Then, $\La_{\Ka}$ is defined weakly by 
\begin{equation}\label{LK}
	(\displaystyle \La_{\Ka}u,v)= 
	\iint_{M\times M} |u(x)-u(y)|^{p-2}(u(x)-u(y))(v(x)-v(y))\Ka(x, y; g)
	\,dv_g(x)dv_g(y).
\end{equation}
We thus look for a non-trivial solution to a non-local equation like
\begin{equation}\label{MainEqu200}
\La_{\Ka}u + h|u|^{p-2}u = f|u|^{q-2}u,
\end{equation}
distinguishing two cases: the critical and sub-critical problems. In the sub-critical case $q<p^*$, an application of the Mountain Pass Theorem gives us the following result.

\begin{thm}\label{subcritical}
 Let $s\in (0,1),\,sp<n$, and $f\ge 0$ and $h$ be smooth functions on $M$. Assuming the coercivity condition 
\begin{equation}\label{coercivity}
	 [u]_{s,p}^p +  \int_M h|u|^p\,dv_g \ge C \|u\|_{s,p}^p \qquad  \text{for any }u\in W^{s,p}(M),
\end{equation}
the equation \eqref{MainEqu200} with $1<q<p^*$ has a non-trivial solution. 
\end{thm}

Let us denote by $K(n,s,2)$ the best constant in the classical embedding $ W^{s,2}(\R^n)\subset L^{2^*}(\R^{n}) $.
For the critical case, we shall prove the existence of a solution to the problem
\begin{equation}\label{MainEquK}
\La_{\Ka}u + hu = f|u|^{p^*-2}u 
\end{equation}
where $h$ and $f$ are smooth functions on $M$, $p=2.$ 
This problem is related to the following fractional Sobolev inequality, which is the main result of this paper.

\begin{thm}\label{TeoBestConst} For any $\eps>0$ there exists $C_\eps>0$ such that 
	\begin{equation}\label{BestConstant} 
		\left( \int_M |u|^{2^*} \,dv_g\right)^\frac{2}{2^*} 
		\le (K(n,s,2)+\eps) \iint_{M\times M} |u(x)-u(y)|^2\Ka(x,y;g)\,dv_g(x)dv_g(y)
		+ C_\eps \int_M u^2\,dv_g 
	\end{equation} 
	holds for any $u\in W^{s,2}(M)$. Moreover, $K(n,s,2)$ is the least possible constant. 
\end{thm} 
In this sense, the present work may be seen as the extension of some classical results
for the Laplacian to the case of non-local fractional operators.
Unfortunately, in the non-local and non-linear case $p\neq 2$, there is no
information about the asymptotic behaviour at infinity of optimizers of the Sobolev
inequality in $\R^n$, and we could not prove this Theorem for $p\neq 2$. 
Recall that, for $p = 2$, the extremals are of the explicit form $cU (\frac{|x-x_0|}{\eps^2})$
with
\begin{equation}\label{bubble}
	U(x) = (1+|x|^2)^{-\frac{n-2s}{2}},
\end{equation}
see \cite{chen2006classification}.
Although it has been conjectured that this extremal has a similar
explicit form for the general case,  it is still an open problem.

Consider the functional 
$J_\Ka:W^{s,2}(M)\to \R$ defined by 
\begin{equation}\label{defJK}
	J_\Ka(u)=\frac{1}{2}\iint_{M\times M}|u(x)-u(y)|^{2}\Ka(x,y;g)dv_g (x)dv_g(y)+\frac{1}{2}\int_{M}h|u(x)|^{2}dv_g(x) 
\end{equation}
that we will minimize over 
\begin{equation}\label{defH}
	 H = \{ u\in W^{s,2}(M):\, \int_M f|u|^{2^*}\,dv_g=1\}.
\end{equation}
The solutions to \eqref{MainEquK}, with $p=2$, can be found as the critical points of $J_\Ka$ restricted to $H$.
As a corollary of the main result, we have the following Theorem: the non-local counterpart of a result well-known in the local setting.  

\begin{thm}\label{ExistsSolution}  Let $f\ge 0$ and $h$ be smooth functions on $M$. Assume the coercivity condition
	\begin{equation}\label{coerc2}
	 J_{\Ka}(u)\ge C\Vert u\Vert_{s,2}^2 \qquad \text{for any }u\in W^{s,2}(M) 
	\end{equation}
	for some positive constant $C$.  If
	\begin{equation}\label{Cond}
	\inf_H J_\Ka < \left(2\left(\max\,f\right)^{2/2^*}K(n,s,2)\right)^{-1},
	\end{equation}
	then the infimum in the l.h.s. of \eqref{Cond} is attained at some nonzero $u_0\in H$. 
	In particular, $u_0$ is a non-trivial solution to \eqref{MainEquK}. 
\end{thm}

The rest of this paper is organized as follows:
In Sect. 2, we set down some notation that we will use throughout the paper and we prove Theorem \ref{FractSpaces}, showing the equivalences between $\widetilde{W}^{s,p}(M)$ and the usual fractional spaces $B^s_{p,q}$ and $F^s_{p,q}$. 
In Sect. 3, we prove Theorem \ref{subcritical}, which gives us the existence of a solution for the sub-critical problem. 
In Sect. 4, we find the optimal Sobolev embedding given by Theorem \ref{TeoBestConst}. 
In Sect. 5, we apply the results of Sect. 4. to
 prove Theorem \ref{ExistsSolution}, which establishes the existence of a non-trivial solution to the problem \eqref{MainEquK}.
 In Sect. 6, we give some technical computations related to the function $U$ defined in \eqref{bubble}.

%
%


\section{Equivalence with the usual fractional spaces}

This section is devoted to defining the fractional Sobolev spaces on
Riemannian manifolds and proving some results related to the Triebel-Lizorkin and Besov spaces. 
For further details on the fractional Sobolev spaces in $\R^n$, we refer to \cite{NPV} and the
references therein.

\subsection{Preliminaries and notation}

Here we collect some
elementary results, which will be helpful in the main estimates of the paper.

Given $(M,\ g)$ a smooth Riemannian manifold and $\gamma$ : $[a,\ b]\rightarrow M$, a curve of class $C^{1}$, 
the length of $\gamma$ is
$$
L(\gamma)=\int_{\alpha}^{b}\sqrt{g(\gamma(t))
\left(\frac{d\gamma(t)}{dt},\frac{d\gamma(t)}{dt}\right)}dt.
$$
For $x, y\in M$ let $\mathcal{C}$ be the space of piecewise $C^1$ curves $\gamma$ : $[a,\ b]\rightarrow M$ such that $\gamma(a)=x$ and $\gamma(b)=y$. Then $d_{g}(x, y)=\inf_{\mathcal{C}}L(\gamma)$ is the distance associated with $g.$
We denote by $dv_{g}(x)= \sqrt{\det(g_{ij})} dx$ the Riemannian volume element on $(M,\ g)$, where the $g_{ij}$ are the components of the Riemannian metric $g$ in the chart and $dx$ is the Lebesgue volume element of $\mathbb{R}^{n}$.

For any $x \in M$ consider the exponential map $\exp_x : T_x M \to M$. Then, we can fix $r  > 0$ 
such that $\exp_x\big|_{B_{r}}: B_{r} \to B_r (x)$ is a diffeomorphism for any $x \in M$. 
Throughout the paper, we will denote by $B_R$ the ball in  $\R^n$ centred
at $0$ with radius $R$ and  by $B_{R}(x)$ the ball in $M$ centred at $x$ with radius $R$ for the
distance $d_g$. Also, we denote by $(\exp_x^*)g$ the metric in $\R^n$ defined as the pullback of $g$ via the exponential map.


We shall need the elementary (see \cite[Lemma 2.53]{lee2018introduction})
\begin{lemma}\label{partitions}
Given $\eps>0$, then there exist a $\delta>0$ smaller than the injectivity radius of $(M,g)$ and a covering of $M$ by balls $\{B_{\delta}(x_i), i=1,\dots,N\}$, such that  for any $i=1,\dots,N$, 
the following properties hold in the exponential chart $(B_{\delta}(x_i),\expt_{x_i}^{-1})$: 
\begin{equation}\label{VolComp} 
(1-\eps) \ dv_\xi\le dv_{\exp_{x_i}^*g} \le (1+\eps) \ dv_\xi, \qquad \text{ and }
\end{equation} 
\begin{equation}\label{DistComp} 
(1-\eps) \ d_\xi(x,y) \le d_g(\expt_{x_i}(x),\expt_{x_i}(y)) \le (1+\eps) \ d_\xi(x,y), 
\end{equation} 
where $\xi$ is the Euclidean metric. 	
\end{lemma}

In the following, for any  $\alpha> 0$, we will say that $I_\eps=O(\eps^\alpha)$  as $\eps\to0$
if there exists a $C > 0$ such that $|I_\eps|\leq C\eps^\alpha$ as $\eps\to 0$.
Additionally, if $(a_i)$ and $(b_i)$ are two real sequences then $a_i= o(b_i)$ means that for any $\varepsilon>0$ and $i$ big enough one has $|a_i|\leq \varepsilon |b_i|$. 

\subsection{Triebel-Lizorkin and Besov spaces on Riemannian manifolds}

Consider a smooth connected complete Riemannian manifold $(M,g)$ with positive injectivity radius
$r_0$. Assume moreover that there exist $c>0$ and, for any multi-index $\alpha$, constants $c_\alpha>0$ such that 
$$ \det(g_{ij})\ge c,\qquad |D^\alpha g_{ij}|\le c_\alpha $$ 
in the normal chart at any point in $M$. 

H. Triebel defines  in \cite[Definition 3]{T2} the Triebel-Lizorkin spaces $F^s_{p,q}(M)$ and the Besov spaces $B^s_{p,q}(M)$ on $M$ using exponential  charts at points $x_j$ covering $M$ and a subordinate partition of unity $(\psi_j)_j$.
For simplicity, we suppose that $p,q\ge 1$ and $s\in\R$. They are defined by 
$$ F^s_{p,q}(M)=\{f\in D'(M):\, \|f\|_{F^s_{p,q}(M)}:=
\left(\sum_j \|\psi_jf\circ \exp_{x_j}\|_{F^s_{p,q}(\R^n)}^p\right)^{1/p}<\infty 
\}, $$  
$$ B^s_{p,q}(M)=(F^{s_0}_{p,p}(M),F^{s_1}_{p,p}(M))_{\theta,q} 
\qquad s=(1-\theta)s_0+\theta s_1, $$ 
where $D'(M)$ denotes the dual space of $C^\infty(M)$. 
He proves in \cite{T2} that the definition of $F^s_{p,q}(M)$ is independent of the choice of the points $x_j$ and the partition of unity and that the definition of $B^s_{p,q}(M)$ is independent of the choice of $s_0$ and $s_1$. Moreover, they are Banach spaces when $p,q\ge 1$.
He then gives  equivalent norms of these spaces to mimic the existing theory in $\R^n$. In particular, he proves that $F^s_{p,p}(M)=B^s_{p,p}(M)$ and that 
\begin{equation}\label{NormEquiv}
\|f\|_{F^s_{p,q}(M)}^p\simeq \sum_j \| \psi_jf\|_{F^s_{p,q}(M)}^p \qquad \text{if } p<\infty
\end{equation}
(and usual modification when $p=\infty$). 
Then in \cite{T}, he proves the characterization of these spaces using finite differences but, as he mentions himself, with unnatural restrictions on the coefficient. In particular, he cannot consider the case $s\in (0,1)$.

\subsection{Equivalence}

Next, we prove that the spaces $\widetilde{W}^{s,p}(M)$ satisfy a property analogous to the one satisfied by the Triebel-Lizorkin spaces $F^s_{p,q}(M)$ described in \eqref{NormEquiv}.

\begin{prop}\label{Prop.Equiv} 
	Let $(M,g)$ be a compact Riemannian manifold that we cover with a finite number of exponential charts at points $\{x_i: i=1, \dots, N\}$. Let $\{\eta_i: i=1, \dots, N\}$ be a partition of unity associated to this covering. There exist constants $C,C'>0$ such that for any $u\in \widetilde{W}^{s,p}(M)$, 
\begin{equation}\label{NormEquiv2}
	C'\|u\|_{s,p}^p\le \sum_{i=1}^N \|\eta_i u\|_{s,p}^p
	\le  C\|u\|_{s,p}^p. 
\end{equation}
\end{prop}

\begin{proof}
	First, we observe that
	\begin{eqnarray*}
		|\eta_iu (x) - \eta_iu (y)|^p &\leq& 2^{p-1} (\eta_i (x)^p |u(x) - u(y)|^p + |u(y)|^p | \eta_i(x) - \eta_i(y)|^p ) \\
		&\leq& 2^{p-1} (|u(x) - u(y)|^p + |u(y)|^p |\eta_i(x) - \eta_i(y)|^p ).
	\end{eqnarray*}
	Then we get
	\begin{eqnarray*}
		&&\iint_{M\times M}\frac{|\eta_iu (x) - \eta_iu (y)|^p }{d_g(x,y)^{n+sp}}\ dv_g(x)dv_g(y) \leq \\
		&&2^{p-1} \left(\iint_{M\times M} \frac{|u(x) - u(y)|^p}{d_g(x,y)^{n+sp}}\ dv_g(x)dv_g(y) + \int_{M}|u(y)|^p \int_{M} \frac{| \eta_i(x) - \eta_i(y)|^p }{d_g(x,y)^{n+sp}}\ dv_g(x)dv_g(y) \right).
	\end{eqnarray*}	
	Since
	\begin{eqnarray*}	
	\int_{M} \frac{| \eta_i(x) - \eta_i(y)|^p }{d_g(x,y)^{n+sp}}\ dv_g(x) < \infty \hspace{0.5cm} \text{ and }	\hspace{0.5cm}
	\int_{M} | \eta_iu(x)|^p \ dv_g(x) <  C \int_{M} | u(x)|^p \ dv_g(x),
	\end{eqnarray*}	
	we have
	\begin{eqnarray*}	
	\sum_{i=1}^N \Vert  \eta_i u \Vert_{s,p}^p \leq C \Vert  u \Vert_{s,p}^p, 
	\end{eqnarray*}	
	with $C>0$ depending on $N,p, |M|$ and $\Vert \eta_i\Vert_\infty$, for all $i=1, \dots, N$.

	On the other hand, we note that $\|u\|_{p}=\|\sum_i \eta_i u\|_{p}\le \sum_i \|\eta_iu\|_{p}$.  
	By Jensen's inequality we have 
	\begin{equation}\label{Jensen1}
	\|u\|_{p}^{p} \le  N^{p-1} \sum_{i=1}^N \|\eta_iu\|_{p}^{p}.
	\end{equation}
	Furthermore, we can write 
	$$
	 |u(x)-u(y)|^p=|\sum_{i=1}^{N}(\eta_iu(x)-\eta_iu(y))|^p
	\leq N^{p-1}\sum_{i=1}^{N}|\eta_iu(x)-\eta_iu(y)|^p
	$$
	and then
	\begin{equation}\label{Jensen2}
		\iint_{ M\times M} \frac{ |u(x)-u(y)|^p}{d_g(x,y)^{n+sp}}\,dv_g(x)dv_g(y)
		\leq C' \sum_{i=1}^N \iint_{  M\times M} \frac{ |\eta_iu(x)-\eta_iu(y)|^p}{d_g(x,y)^{n+sp}}\,dv_g(x)dv_g(y).
	\end{equation}
	Thus, from \eqref{Jensen1} and \eqref{Jensen2}, it follows
	$$
	\Vert  u \Vert_{s,p}^p = [ u]_{s,p}^p  + \|u\|_{p}^{p}  \leq
     C' \sum_{i=1}^N ( [\eta_i u]_{s,p}^p  + \|\eta_iu\|_{p}^{p} ) = C'	\sum_{i=1}^N \Vert   u \Vert_{s,p}^p. 
	$$
	with $C'>0$ depending on $N,p$.
\end{proof}

As a corollary of Proposition \ref{Prop.Equiv} we can prove the Theorem \ref{FractSpaces}.

\begin{proofEquiv}
	\end{proofEquiv} 
For $\eps>0$, let $\{\eta_i,\, i=1,\dots,N\}$ be a partition of unity adapted to the covering 
$\{B_{\delta}(x_i),\,i=1,\dots,N\}$ given by Lemma \eqref{partitions} and let $v_i=(\eta_iu)\circ \expt_{x_i}$.
 In order to prove the Theorem, we have to check first that 
$\|\eta_j u\|_{s,p} 
\simeq \|v_j\|_{s,p}$.
Indeed we have 
$$ 
	(1-\eps)  \int_{\R^n} |v_i|^{p} \, dv_\xi \le \int_M |\eta_iu|^{p} \,dv_g 
	=\int_{\R^n} |v_i|^{p} \, dv_{\expt_{x_i}^{*}g} \\
	 \le  (1+\eps)  \int_{\R^n} |v_i|^{p} \, dv_\xi.
$$
So we have 
\begin{equation}\label{part.norm}
 (1-\eps) \|v_i\|_{p}^p\leq \|\eta_i u\|_{p}^p\leq (1+\eps) \|v_i\|_{p}^p.
\end{equation}

Now we shall estimate $	[\eta_i u]_{s,p}$.
Denote $U_i:=\expt_{x_i}^{-1}(B_{\delta}(x_i))$, $i=1,\dots,N$. 
Since $\text{supp}(v_i) \subset U_i$, 
we can write 
\begin{equation}\label{part}
	[v_i]_{s,p}^p =  2\iint_{ x\in U_i,y\not\in U_i} \frac{ |v_i(x)-v_i(y)|^p}{|x-y|^{n+sp}}\,dxdy + 
	\iint_{ x\in U_i,y\in U_i} \frac{ |v_i(x)-v_i(y)|^p}{|x-y|^{n+sp}}\,dxdy.
\end{equation} 

Recall that $\text{supp}\,\eta_i$ is a compact subset of $B_{\delta}(x_i)$. 
Denote $K_i:=\expt_{x_i}^{-1}(\text{supp}(\eta_i))$. 
Since $K_i$ is compact, we can take some $\alpha>0$ such that 
$$ |x-y|\ge \alpha>0 \qquad \text{for any $x\in K_i$, $y\in \R^n\backslash U_i$, $i=1,\dots,N$.}$$ 
It follows that 
\begin{eqnarray*} 
	\iint_{x\in U_i,\,y\not\in U_i}  \frac{ |v_i(x)-v_i(y)|^p}{|x-y|^{n+sp}}\,dxdy 
	& = & \iint_{ x\in K_i,\,y\not\in U_i} \frac{ |v_i(x)|^p}{|x-y|^{n+sp}}\,dxdy \\  
	& \le & C_{\alpha,n} \int_{K_i} |v_i(x)|^p\,dx   \\
	& \le & C_{\alpha,n,\eps} \int_M \eta_i |u|^p\,dv_g. 
\end{eqnarray*} 
	
To bound the 2nd integral in \eqref{part}, we write

\begin{eqnarray*} 
		\iint_{ x\in U_i,y\in U_i} \frac{|v_i(x)-v_i(y)|^p}{|x-y|^{n+sp}}\,dxdy 
& \leq&C_\eps\iint_{x\in U_i,y\in U_i}\frac{|v_i(x)-v_i(y)|^p}{|x-y|^{n+sp}}\,dv_{\exp^*_{x_i}g}(x)dv_{\exp^*_{x_i}g}(y)\\
& = & C_\eps \iint_{M \times M}\frac{|(\eta_iu)(x)-(\eta_iu)(y)|^p}{d_g(x,y)^{n+sp}}\,dv_g(x)dv_g(y).\\
\end{eqnarray*} 
Then we have $
[v_i]_{s,p}^p \leq   C_{\eps,\alpha} [\eta_i u]_{s,p}^p.
$
In the same way, it can be proved that
$$
[v_i]_{s,p}^p \geq   C'_{\eps,\alpha} [\eta_i u]_{s,p}^p
$$
for some constant $C'_{\eps,\alpha}$. Therefore we have that the norms are equivalent and 
using the fact that $\widetilde{W}^{s,p}(\R^n)=B^s_{p,p}(\R^n)=F^s_{p,p}(\R^n)$ we conclude with 
\eqref{NormEquiv}. 
\qed

From now on, we will simply denote 	$ \widetilde{W}^{s,p}(M)=W^{s,p}(M)$.  
It follows from the previous Theorem and the general theory developped by H. Triebel that the spaces $W^{s,p}(M)$ satisfy all the usual properties. We describe some of them in the following Proposition.
%
%
\begin{prop}  For $s\in(0,1)$ and $0<sp<n$, we have
\begin{enumerate}[(i)]
	\item $W^{s,p}(M)$ is a Banach space with the norm $\| \cdot \|_{s,p}$.
	\item If $p>1$ the space $W^{s,p}(M)$ is reflexive.
	\item The space $C^\infty(M)$ of smooth functions is dense in $W^{s,p}(M).$
	\item The embedding  $W^{s,p}(M)\hookrightarrow L^q(M)$ is continuous when $q\le p^*$ and compact when $q < p^*$. 
\end{enumerate}

\end{prop}

%
%

%

\section{Proof of Theorem \ref{subcritical}}

In this section, we shall prove the existence of a solution to the problem
\[
\La_{\Ka}u + h|u|^{p-2}u = f|u|^{q-2}u. 
\]
where  $1<q<p^*$, the functions $h,f$ satisfy the coercivity condition \eqref{coercivity} and $\La_{\Ka}$ is the non-local operator defined in \eqref{LK}. 


Recall that $\Ka( \cdot, \cdot \ ;g):(M\times M)\backslash \Di\rightarrow(0,\ +\infty)$ is a function which satisfies \ref{K1}-\ref{K4}.
The first three conditions for $\Ka$ are standard (see, e.g. \cite{KKP,P,SV2013}) and condition \ref{K4} is
necessary to prove our present result. In \cite[Assumption 2.1]{frank2008non}, the authors consider a family of
measurable functions $k_\eps$ , for $\eps > 0$, on  $\R^n \times \R^n$ satisfying similar properties to $\eps^{n+sp}\tilde \Ka(X,Y,\eps^2 g_\eps)$. A model for $\Ka$ is given by $\Ka(x,y;g)=d_g(x,y)^{-(n+ps)} + d_g(x,y)^{-\alpha}$
with $\alpha\in(0,n+ps)$. The operator  $\La_{\Ka}$ has been studied for particular values of $\alpha$ in \cite{alonso2018integral}  and \cite{dNS}. 
In what follows, we give a typical example for the kernel $\Ka$.

\begin{example}
	Look at the classical fractional kernel given by
	$$
	\Ka_0(x,y;g)=d_g(x,y)^{-(n+ps)}+d_g(x,y)^{-\alpha}
	$$
	with $\alpha\in(0,n+ps)$. For a particular $\alpha$, the operator $\La_{\Ka_0}$ is the fractional Laplacian operator $(-\Delta_{p,g})^s$.
	Then, $\Ka_0$ trivially satisfies conditions \eqref{K2} and \eqref{K3}. Let us now check the third condition. Indeed,
	\begin{eqnarray*}
				\tilde \Ka_0(X,Y;\eps^2 g_\eps)&=& \tilde \Ka_0(X,Y;T^*\exp_{x_0}^* g)
			\\
			&=& \Ka_0(\exp_{x_0}(\eps X),\exp_{x_0}(\eps Y); g)
			\\
			&=& d_g(\exp_{x_0}(\eps X), \exp_{x_0}(\eps Y))^{-(n+ps)}+d_g(\exp_{x_0}(\eps X), \exp_{x_0}(\eps Y))^{-\alpha}
			\\
			&=& d_{\exp_{x_0}^*g}(\eps X, \eps Y)^{-(n+ps)}+d_{\exp_{x_0}^*g}(\eps X, \eps Y)^{-\alpha}
	\end{eqnarray*}
	because $exp_{x_0}$ is an isometry from the ball 
	$(B_\delta ,(exp_{x_0})^*g)$ in $T_{x_0}M$ to the ball $(B_\delta(x_0),g)$ in $M$. Finally, from 
	$$
	\eps^{(n+ps)}\tilde \Ka_0(X,Y;\eps^2 g_\eps)=  d_{\exp_{x_0}^*g}(X, Y)^{-(n+ps)}+ \eps^{n+ps-\alpha}d_{\exp_{x_0}^*g}(X, Y)^{-\alpha}, 
	$$ 
	and Lemma \ref{partitions}, it follows that \eqref{K4} holds for $\Ka_0$. 
\end{example}

In the following, we shall use  the Mountain Pass Theorem to prove Theorem \ref{subcritical}, where $p<q.$
By definition, a sequence $\left(u_{i}\right)$ of functions in $W^{s,p}(M)$ is said to be a Palais-Smale sequence for $I$ if:
\begin{enumerate}[(PS1)]
	\item $I\left(u_{i}\right)$ is bounded, and
	\item $I' \left(u_{i}\right) \rightarrow 0$ in $W^{s,p}(M)'$ as $i \to +\infty$, where $W^{s, p}(M)'$ denotes the dual space of $W^{s, p}(M)$.
\end{enumerate}
We say that $I$ satisfies (PS) condition in $W^{s, p}(M)$, if for any Palais-Smale sequence $\left\{u_{i}\right\} \subset W^{s, p}(M)$, there exists a convergent subsequence of $\left\{u_{i}\right\}$.
Recall

\begin{MPL} Let $I$ be a $C^{1}$ function on a Banach space $E$. Suppose that $I$ satisfies the Palais-Smale condition. Suppose also
	\begin{enumerate}[(MP1)]
		\item\label{MP1} $I(0)=0$,  
		\item\label{MP2} there exist constants $\rho, r$ such that
		$
		I(u) \geq \rho \ \text{ for all } u \in \partial B_{0}(r)\subset E, \text{ and }
		$
		\item\label{MP3} there exists an element $u_0\in E$ with
		$
		I\left(u_{0}\right)<\rho.
		$
	\end{enumerate}
	Let	
	$$
	c=\inf _{\gamma \in \Gamma} \max _{u \in \gamma} \Phi(u)
	$$
	where $\Gamma$ stands for the class of continuous paths joining 0 to $u_{0}$. Then $c$ is a critical value of $I$.	
\end{MPL}

Let 
$J_\Ka:W^{s,p}(M)\to \R$ defined by
\[
J_\Ka(u)=\frac{1}{p}\iint_{M\times M}|u(x)-u(y)|^{p}\Ka(x,y;g)dv_g (x)dv_g(y)+\frac{1}{p}\int_{M}h|u(x)|^{p}dv_g(x),
\]
and let $I:W^{s,p}(M)\to \R$ be the energy functional associated with the problem:
\[
I(u)= J_\Ka(u) - \frac{1}{q}\int_M f|u|^{q}\,dv_g.
\]

\begin{lemma}\label{I esC1}
		$I\in C^{1}(W^{s, p}(M),\R)$ and 
			\[
		\left\langle I^{\prime}(u), v\right\rangle=\left\langle \La_\Ka u,v\right\rangle
		+\int_M h|u|^{p-2} u v \ dv_g - \int_M f|u|^{q-2} u v \ dv_g
		\]
\end{lemma}
We shall prove the Lemma in two steps.
	\begin{step1}
		The functional $J_\Ka \in C^{1}\left(W^{s, p}(M), \mathbb{R}\right)$ and
		\[
		\left\langle J_\Ka^{\prime}(u), v\right\rangle=\left\langle \La_\Ka u,v\right\rangle
		+\int_M h|u|^{p-2} u v \ dv_g
		\]
		for all $u, v \in W^{s, p}(M) .$ Moreover, if $u \in W^{s, p}(M)$, then $J_\Ka^{\prime}(u) \in W^{s, p}(M)'$.
	\end{step1}
	
	\begin{proof}
		Firstly, it is easy to see that for all $u, v \in W^{s, p}(M)$, it holds 
		\[
		\left\langle J_\Ka^{\prime}(u), v\right\rangle=\left\langle \La_\Ka u,v\right\rangle
		+\int_M h|u|^{p-2}uv \ dv_g.
		\]
		Then, it follows that $J_\Ka^{\prime}(u) \in$ $W^{s, p}(M)'$  for each $u \in W^{s, p}(M).$
		Next, we prove that $J_\Ka \in C^{1}\left(W^{s, p}(M), \mathbb{R}\right) .$ Let $\left\{u_{j}\right\} \subset W^{s, p}(M)$ a sequence such that $u_{j} \rightarrow u$ for some $u \in W^{s, p}(M)$, strongly in $W^{s, p}(M)$ as $n \rightarrow \infty.$
Now, using H\"older's inequality, we have
	\begin{eqnarray*}
			\left\langle \La_\Ka u_j,v\right\rangle &\leq&
			\iint_{M\times M}\left|u_{j}(x)-u_{j}(y)\right|^{p-1}\Ka(x,y;g)^{\frac{p-1}{p}}(v(x)-v(y))\Ka(x,y;g)^{\frac{1}{p}} dv_g(x) dv_g(y)\\
			&\leq&\left(\iint_{M\times M}\left|u_{j}(x)-u_{j}(y)\right|^{p}\Ka(x,y;g)dv_g(x) dv_g(y)\right)\\
			&\times& \left(\iint_{M\times M}\left|v(x)-v(y)\right|^{p}\Ka(x,y;g)dv_g(x) dv_g(y)\right)\\
		\end{eqnarray*}
	From there, we get
\begin{equation}\label{cont1}
	\left\langle \La_\Ka u_j,v\right\rangle 	\leq C [u_j]^p_{s,p} [v]^p_{s,p}
\end{equation}
for $v\in W^{s,p}(M)$.
		Additionally, the fact that $u_{j} \to u$ strongly in $W^{s, p}(M)$ implies
		\begin{equation}\label{cont2}
			\lim _{n \rightarrow \infty} \int_{M}\left(h|u_{j}(x)|^{p}-h|u(x)|^{p}\right) \, dv_g(x) =0.
		\end{equation}
		Combining \eqref{cont1} and \eqref{cont2}, we have
		\[
		\left\|J_\Ka^{\prime}\left(u_{j}\right)-J_\Ka^{\prime}(u)\right\|=\sup _{v \in W_{0}^{s, p}(M),\|v\|_{s, p} \leq 1}\left|\left\langle J_\Ka^{\prime}\left(u_{j}\right)-J_\Ka^{\prime}(u), v\right\rangle\right| \rightarrow 0
		\]
		as $n \rightarrow \infty.$
	\end{proof}
	Using the same strategy as in the previous Step, we have
	
	\begin{step2} If we define
		\[
		H(u)=	\frac{1}{q} \int_M f|u|^{q}\,dv_g,
		\]
		then $H \in C^{1}\left(W^{s,p}(M), \mathbb{R}\right)$ and
		\[
		\left\langle H^{\prime}(u), v\right\rangle= \int_M f|u|^{q-2} u v \ dv_g
		\]
		for all $u, v \in W^{s, p}(M).$
	\end{step2}
	Due to Steps 1 and 2, critical points of $I:W^{s,p}(M)\to \R$ are weak solutions to the problem \eqref{MainEqu200}. 
	We shall study the cases $p<q$ and $q<p$ separately.
We intend to apply the Mountain Pass Theorem to $I$ for the first case  with $E=W^{s,p}(M)$. 


\begin{SubcriticalProof1}
\end{SubcriticalProof1}
	\begin{step1}
		$I$ satisfies the Palais-Smale condition.
	\end{step1}
	\begin{proof}
		Let $\left(u_i\right)_i\subset W^{s,p}(M)$ be a Palais-Smale sequence for $I$ i.e. $I(u_i)=O(1)$ and $I'(u_i)=o(1)$. 
		Then 
		$$ o(\|u_i\|_{s,p})\|u_i\|_{s,p} + O(1) =\left\langle I'(u_i), \, u_i \right\rangle- qI(u_i) = (1-q/p)J_\Ka(u_i). $$ 
		Because of the coercivity assumption \eqref{coercivity}, we have $J_\Ka(u_i)\ge C\|u_i\|_{s,p}^p$. Since $p\neq q$, we obtain 
		$\|u_i\|_{s,p}^p=o(\|u_i\|_{s,p})\|u_i\|_{s,p} + O(1)$ from which we deduce that  $(u_i)$ is bounded in $W^{s,p}(M)$.
		We can thus extract from $(u_i)_i$ a subsequence converging to some $u\in W^{s,p}(M)$ weakly in $W^{s,p}(M)$ and strongly in $L^p(M)$ and $L^q(M)$ 
		(because $q<p^*$). Using this strong convergence and denoting $v_i:=u_i-u$, it is easily seen that 
	   $$ o(1) = \left\langle I'(u_i), v_i \right\rangle= \iint |u_i(x)-u_i(y)|^{p-2}(u_i(x)-u_i(y))(v_i(x)-v_i(y))K(x,y,g) dv_g(x)dv_g(y) \, + \, o(1). $$ 
		Moreover, the weak convergence $v_i\to 0$ in $ W^{s,p}(M)$ gives that 
	    $$ \iint |u(x)-u(y)|^{p-2}(u(x)-u(y))(v_i(x)-v_i(y))K(x,y,g) dv_g(x)dv_g(y) = o(1). $$
		Thus 
        \[ \iint \big(|u_i(x)-u_i(y)|^{p-2}(u_i(x)-u_i(y)) -  |u(x)-u(y)|^{p-2}(u(x)-u(y))\big)
        \]
        \[
		\times\big(v_i(x)-v_i(y)\big)K(x,y,g) dv_g(x)dv_g(y) = o(1).  
		\]
	Applying the classical inequality (see \cite{adams1999function})
    $$ (|a|^{p-2}a - |b|^{p-2}b)(a-b) \ge C|a-b|^p \qquad a,b\in \R , p>1$$ 
    with $a=u_i(x)-u_i(y)$ and $b=u(x)-u(y)$ we obtain 
   $$ \iint |v_i(x)-v_i(y)|^p K(x,y,g) dv_g(x)dv_g(y) = o(1) $$ 
    so that $[v_i]_{s,p}\to 0$. Since $v_i\to 0$ in $L^p$, we deduce that $v_i\to 0$ strongly in $ W^{s,p}(M)$. 
	\end{proof}
	We finally verify the remaining hypothesis of the Mountain Pass Theorem.
	\begin{step2}
		$I$ satisfies conditions (MP\ref{MP1}) to (MP\ref{MP3}).
	\end{step2}
	\begin{proof}
		Since $J_\Ka$ is coercive, and thanks to the Sobolev inequality corresponding to the embedding of $W^{s,p}(M)$ in $L^{q}(M)$, there exists positive constants $C_{1}, C_{2}>0$ such that for any $u \in W^{s,p}(M)$,
		\begin{equation}
			I(u)= J_\Ka(u) - \int_M f|u|^{q}\,dv_g \ \geq \ C\left( \|u\|_{s,p}^{p} - \|u\|_q^q\right)
			\ \geq \ C\left( \|u\|_{s,p}^{p}- \|u\|_{s,p}^{q}\right)
		\end{equation}
		
		Taking $r>0$ small enough, then it follows that there exists $\rho>0$ such that for any $u \in \partial B_{0}(r)$, $I(u) \geq \rho$. Independently, $I(0)=0$, while for $v_{0} \in W^{s,p}(M), v_{0} \not \equiv 0$,
		\[
		\lim _{t \rightarrow+\infty} I\left(t v_{0}\right)=\lim _{t \rightarrow+\infty} \left( t^pJ_\Ka(v_0) - \frac{t^q}{q}\int_M f|v_0|^{q}\,dv_g\right)=-\infty
		\]
		It follows that there exists $r>0, \rho>0$, and $u_{0}=t v_{0}$ such that $ 
		I\left(u_{0}\right)<\rho$, $u_{0} \in W^{s,p}(M) \backslash B_{0}(r)$. The Mountain Pass Lemma then gives the existence of a critical point $u_0$ of $I$
		such that
		\[
		I(u_0) > 0=I(0)
		\]
		so that $u_0\neq 0.$
	Thus, the assertion of Theorem \ref{subcritical} follows.	
		\end{proof}
	

\begin{SubcriticalProof2}
\end{SubcriticalProof2}
	We shall show that the functional $I$ is weakly lower semi-continuous.
Let $\left\{u_{i}\right\} \subset W^{s, p}(M)$, such that $u_{i} \rightarrow u$ weakly in $W^{s, p}(M)$ as $n \rightarrow \infty$. Thus, we get that $u_{i} \rightarrow u$ strongly in $L^{q}(M)$.  Due to Lemma \ref{I esC1}, we have the following inequality

$$
I(u_{i})>I(u)+\left\langle\left(I^{\prime}(u), u_{i}-u\right)\right\rangle .
$$
Then we get that $I(u) \leq \liminf _{n \rightarrow \infty}I\left(u_{i}\right)$, ie, $I$ is weakly lower semi-continuous in $W^{s, p}(M)$.
On the other hand, we have $I(t v_0) \to +\infty$ as $t\to +\infty$  for any $v_{0} \in W^{s,p}(M), v_{0} \not \equiv 0$. Since $I$ is weakly lower semi-continuous, it has a minimum point in $W^{s,p}(M)$.
\qed

\section{Optimal Sobolev embedding for p=2 in the critical case}

It is well-known (see e.g. \cite{NPV}) that the following fractional Sobolev embedding 
holds: there exists a constant  $A>0$ such that 
\begin{equation}\label{IneqRn} 
\left( \int_{\R^n} |v|^{p^*}\,dx\right)^\frac{p}{p^*} 
\le A \iint_{\R^n\times\R^n} \frac{|v(x)-v(y)|^p}{|x-y|^{n+sp}}\,dxdy 
\end{equation} 
for any $v:\R^n\to \R$ measurable and compactly supported. 
We denote by $K(n,s,p)$ the best constant in this embedding, namely 
\begin{equation}\label{BestConstSob}
K(n,s,p)^{-1} = \inf_{u\in W^{s,p}(\R^n)} \frac{[u]_{s,p}^p}{\|u\|_{p^*}^{p}}. 
\end{equation}
In order to prove our main result, we will consider $p=2$ from now on.
\begin{demTeo}
	
\begin{step1}
	Suppose that there are constants $C_1,C_2>0$ such that 
	\begin{equation}\label{BestConstantGen} 
	\left( \int_M |u|^{2^*} \,dv_g\right)^\frac{2}{2^*} 
	\le C_1 \iint_{M\times M} |u(x)-u(y)|^2\Ka(x,y;g)\,dv_g(x)dv_g(y)
	\ + \ C_2 \int_M u^2\,dv_g 
	\end{equation} 
	for any $u\in W^{s,p}(M)$. Then $C_1\ge K(n,s,2)$. 
\end{step1}

\begin{proof}
	Let $\eta:[0,+\infty)\to [0,1]$ be a smooth test-function with compact support in $[0,2\delta]$ and such that $\eta\equiv 1$ in $[0,\delta]$. We choose $\delta>0$
such that $2\delta$ is smaller than the injectivity radius of $(M,g)$.  
Given a point $x_0\in M$, $\eps>0$ and $U\in W^{s,2}(\R^n)$ given by 
\begin{equation}
U(x) = (1+|x|^2)^{-\frac{n-2s}{2}}
\end{equation} 
we consider the test-function 
$$
u_\eps(x)=\eta(d_g(x_0,x))U_\eps(x) \qquad \text{where} \qquad
U_\eps(x)=\eps^{-\frac{n-2s}{2}}  U\left( \frac{1}{\eps} exp_{x_0}^{-1}(x) \right).
$$ 
Applying \eqref{BestConstantGen} to $u_\eps$ we obtain 
$$
	\left( \int_M |u_\eps|^{2^*} \,dv_g\right)^\frac{2}{2^*} 
	\le C_1 \iint_{M\times M} |u_\eps(x)-u_\eps(y)|^2\Ka(x,y;g)\,dv_g(x)dv_g(y)
	\ + \ C_2 \int_M u_\eps^2\,dv_g. 
$$
In view of \eqref{CompGrad}, \eqref{BoundL2} and \eqref{BoundL2*}, we can send $\eps\to 0$ to obtain 
$$ \left( \int_{\R^n} |U|^{2^*} \,dx\right)^\frac{2}{2^*} 
\le  C_1 \iint_{\R^n\times \R^n} \frac{|U(x)-U(y)|^2}{|x-y|^{n+2s}}\,dxdy
$$ 
Since $U$ is an extremal for \eqref{BestConstSob} we obtain $C_1\ge K(n,s,2)$. 
\end{proof}


\begin{step2}
Inequality \eqref{BestConstant}  holds. 
\end{step2}

\begin{proof}
	Given $\eps>0$, we take $\delta>0$ smaller than the injectivity radius of $(M,g)$ and a covering of $M$ by balls $\{B_{\delta}(x_i), i=1,\dots,N\}$, such that  for any $i=1,\dots,N$, the properties \eqref{VolComp} and \eqref{DistComp} hold.
	Let $\{\eta_i,\, i=1,\dots,N\}$ be a partition of unity adapted to the covering 
	$\{B_{\delta}(x_i),\,i=1,\dots,N\}$. 
	Then $\|u\|_{2^*}=\|\sum_i \eta_i u\|_{2^*}\le \sum_i \|\eta_iu\|_{2^*}$ and, by Jensen's inequality, we obtain 
	$$ \|u\|_{2^*}^2 \le \frac{1}{N}  \sum_i \|\eta_iu\|_{2^*}^2. $$ 
	We now estimate $\|\eta_iu\|_{2^*}$. Let $v_i=(\eta_iu)\circ \expt_{x_i}$. Then by \eqref{part.norm}, we have
$$
	\int_M |\eta_iu|^{2^*} \,dv_g \le  (1+\eps)  \int_{\R^n} |v_i|^{2^*} \, dv_\xi.
$$
Using \eqref{IneqRn}, we obtain 
\begin{eqnarray*} 
	\|\eta_iu\|_{2^*}^2
	& \le & (1+\eps)A \iint_{\R^n \times \R^n}  \frac{ |v_i(x)-v_i(y)|^2}{|x-y|^{n+2s}}\,dxdy.
\end{eqnarray*} 
		Given $\delta'>0$ small to be specified later, we write the integral in the rhs as 
\begin{align*}
	 \iint_{\R^n \times \R^n}  \frac{ |v_i(x)-v_i(y)|^2}{|x-y|^{n+2s}}\,dxdy  
	&=& \iint_{|x-y|>\delta'} \frac{ |v_i(x)-v_i(y)|^2}{|x-y|^{n+2s}}\,dxdy  \\
	 &+& \iint_{|x-y|<\delta'} \frac{ |v_i(x)-v_i(y)|^2}{|x-y|^{n+2s}}\,dxdy.
\end{align*}
Using that $(a+b)^2\le 2(a^2+b^2)$, we can bound the first integral in the r.h.s. by   
\begin{eqnarray*} 
	\iint_{|x-y|>\delta'}  \frac{ |v_i(x)-v_i(y)|^2}{|x-y|^{n+2s}}\,dxdy  
	& \le & 2 \int_{\R^n} |v_i(x)|^2\left( \int_{|x-y|>\delta'} \frac{dy}{|x-y|^{n+2s}}\right) \,dx \\
	& \le & C_{i}(\delta')^{-2s} \int_{\R^n} |v_i(x)|^2\,dx \\ 
	& \le & C_{n,\eps}(\delta')^{-2s} \int_M \eta_i |u|^2\,dv_g.  
\end{eqnarray*} 		
Thus 

\begin{eqnarray*} 
	\|\eta_iu\|_{2^*}^2
	\le C_{n,\eps}(\delta')^{-2s} \int_M \eta_i |u|^2\,dv_g + 
	(1+\eps)A \iint_{|x-y|<\delta'}  \frac{ |v_i(x)-v_i(y)|^2}{|x-y|^{n+2s}}\,dxdy.
\end{eqnarray*} 	
Denote $U_i:=\expt_{x_i}^{-1}(B_{\delta}(x_i))$, $i=1,\dots,N$. 
Noticing that $v_i(x)=0$ if $x\not\in U_i$, 
we can write the second integral in the r.h.s. as 
\begin{eqnarray*} 
	&& \iint_{|x-y|<\delta'}  \frac{ |v_i(x)-v_i(y)|^2}{|x-y|^{n+2s}}\,dxdy 
	 = \iint_{|x-y|<\delta',\, x\in U_i,y\not\in U_i} \frac{ |v_i(x)-v_i(y)|^2}{|x-y|^{n+2s}}\,dxdy   \\
	&& + \iint_{|x-y|<\delta',\, x\not\in U_i,y\in U_i} \frac{ |v_i(x)-v_i(y)|^2}{|x-y|^{n+2s}}\,dxdy  
	+ \iint_{|x-y|<\delta',\, x,y\in U_i} \frac{ |v_i(x)-v_i(y)|^2}{|x-y|^{n+2s}}\,dxdy    \\ 
	&& = 2\iint_{|x-y|<\delta',\, x\in U_i,y\not\in U_i} \frac{ |v_i(x)-v_i(y)|^2}{|x-y|^{n+2s}}\,dxdy   
	+ \iint_{|x-y|<\delta',\, x,y\in U_i} \frac{ |v_i(x)-v_i(y)|^2}{|x-y|^{n+2s}}\,dxdy.
\end{eqnarray*} 
		Recall that $\text{supp}\,\eta_i$ is a compact subset of $B_{\delta}(x_i)$. 
Denote $K_i:=\expt_{x_i}^{-1}(\text{supp}(\eta_i))$. 
Since $K_i$ is a compact, we can take some $\alpha>0$ such that 
$$ |x-y|\ge \alpha>0 \qquad \text{for any $x\in K_i$, $y\in \R^n\backslash U_i$, and any $i=1,\dots,N$.}$$ 
\begin{eqnarray*} 
	\iint_{|x-y|<\delta',\, x\in U_i,\,y\not\in U_i}  \frac{ |v_i(x)-v_i(y)|^2}{|x-y|^{n+2s}}\,dxdy 
	& = & \iint_{|x-y|<\delta',\, x\in K_i,\,y\not\in U_i}  
	\frac{ |v_i(x)|^2}{|x-y|^{n+2s}}\,dxdy \\  
	& \le & \int_{K_i} |v_i(x)|^2\left( \int_{|x-y|<\delta'} \frac{dy}{\alpha^{n+2s}}\right)\,dx \\
	& \le & C_{\alpha,n,\delta'} \int_{K_i} |v_i(x)|^2\,dx   \\
	& \le & C_{\alpha,n,\delta',\eps} \int_M \eta_i |u|^2\,dv_g. 
\end{eqnarray*} 
We thus obtain 
\begin{eqnarray*} 
	\|\eta_iu\|_{2^*}^2 
	& \le & C_{\alpha,n,p,\delta',\eps} \int_M \eta_i |u|^2\,dv_g 
	+ \ (1+\eps)A \iint_{|x-y|<\delta',\,x,y\in U_i}  
	\frac{ |v(x)-v(y)|^2}{|x-y|^{n+2s}}\,dxdy.  
\end{eqnarray*}

To bound the 2nd integral in the r.h.s. we write 
\begin{eqnarray*} 
	v_i(x)-v_i(y) & = & \eta_i(\expt_{x_i}(x))[ u(\expt_{x_i}(x))-u(\expt_{x_i}(y)) ]  \\
	&&     + u(\expt_{x_i}(y)) [\eta_i(\expt_{x_i}(x)) - \eta_i(\expt_{x_i}(y)) ].  
\end{eqnarray*} 										
Using the inequality $(a+b)^2\le (1+\eps)a^2+C_\eps b^2$, $a,b\ge 0$, we deduce 
\begin{eqnarray*} 
	&& \iint_{|x-y|<\delta',\,x,y\in U_i}  \frac{ |v_i(x)-v_i(y)|^2}{|x-y|^{n+2s}}\,dxdy \\
	&  & \le (1+\eps) \iint_{|x-y|<\delta',\,x,y\in U_i} | \eta_i(\expt_{x_i}(x)) |^2
	\frac{ | u(\expt_{x_i}(x))-u(\expt_{x_i}(y))  |^2}{|x-y|^{n+2s}}\,dxdy \\ 
	&& + C_\eps \iint_{|x-y|<\delta',\,x,y\in U_i}  \frac{ | \eta_i(\expt_{x_i}(x)) - \eta_i(\expt_{x_i}(y))  |^2}{|x-y|^{n+2s}}
	| u(\expt_{x_i}(y))|^2     \,dxdy  \\
	&  =: & I + II. 
\end{eqnarray*} 
Using that $\eta_i\circ \expt_{x_i}$ is Lipschitz with a Lipschitz constant that can be chosen to depend on $\delta$, and thus on $\eps$, and not on $i$, we can bound II: 
\begin{eqnarray*} 
	II 
	& \le & C_{\eps,\delta} \int_{U_i}  | u(\expt_{x_i}(y))|^2    \left( \int_{|x-y|<\delta'} \frac{dx}{|x-y|^{n-(1-s)2}}\right)\,dy \\
	& \le & \frac{C_{\eps,\delta} }{1-s} \int_{U_i}  | u(\expt_{x_i}(y))|^2\,dx \\
	&  \le & C_{\eps,\delta,s} \int_{B_{\delta}(x_i)}  | u|^2\,dv_g.  
\end{eqnarray*}  
	Concerning $I$, 
\begin{eqnarray*} 
	I 
	& \le & (1+C\eps) \iint_{x,y\in B_{\delta}(x_i),\,|\expt_{x_i}^{-1}(x)-\expt_{x_i}^{-1}(y)|<\delta' }
	\eta_i(x)   \frac{|u(x)-u(y)|^2}{|\expt_{x_i}^{-1}(x)-\expt_{x_i}^{-1}(y)|^{n+2s}} dv_g(x)dv_g(y) \\ 
	& \le & 		(1+C\eps)   \iint_{x,y\in B_{\delta}(x_i)} \eta_i(x)  \frac{|u(x)-u(y)|^2}{d_g(x,y)^{n+2s}}\, dv_g(x)dv_g(y) \\
	& \le & 	\Lambda	(1+C\eps)   \iint_{x,y\in B_{\delta}(x_i)} \eta_i(x)|u(x)-u(y)|^2\Ka(x,y;g)\, dv_g(x)dv_g(y)
\end{eqnarray*} 
where we used \eqref{DistComp} and \eqref{K3}. 

We thus obtain 
\begin{eqnarray*} 
	\|\eta_iu\|_{2^*}^2 
	& \le & C \int_M \eta_i |u|^2\,dv_g + C \int_{B_{\delta}(x_i)}  | u|^2\,dv_g  \\
	&&     + \	\Lambda (1+C\eps)A   \int_M \eta_i(x)\left(\int_M |u(x)-u(y)|^2\Ka(x,y;g)\, dv_g(y)\right)\,dv_g(x)
\end{eqnarray*} 	
where $C=C(n,\alpha,s,\delta,\eps)$. 
Since the balls $\{B_{\delta}(x_i)$, $ i=1,\dots,N\}$ overlap  a finite number of times, summing this inequality over $i=1,\dots,N$ gives \eqref{BestConstant}.  
\end{proof}

\end{demTeo}

\section{Applications to equations with a fractional non-local operator}
In this section, we prove Theorem \ref{ExistsSolution}.
The proof is similar to the classical one in the local case (see, e.g. \cite{hebey}), but we prefer to give a full demonstration as some technical differences arise. 

We consider the critical equation \eqref{MainEquK} for $p=2$, which is
\begin{equation}\label{MainEquK2}
\La_{\Ka}u + hu = f|u|^{2^*-2}u.
\end{equation}
We are interested in the weak formulation of that equation given by the following problem:
$$
\iint_{M\times M} (u(x)-u(y))(v(x)-v(y))\Ka(x, y; g) \,dv_g(x)dv_g(y) \ + \ \int_M huv = \int_M f|u|^{2^*-2} uv dv_g(x),
$$
for all $v\in W^{s,2}(M),$ and $u\in  W^{s,2}(M).$
Critical points of $J_{\Ka}$ in $H$ are solutions to the problem \eqref{MainEquK2}, where $J_{\Ka}$ and $H$ are defined in \eqref{defJK} and \eqref{defH}.

\begin{ExistenceProof}
\end{ExistenceProof} 
	Denote $\mu_0:=\inf_{u\in H} J_\Ka(u)$. We shall prove that there is a $u_0\in H$ which attained the infimum in $\mu_0$.
	We approximate the minimization problem by the subcritical problem 
	$\mu_q:=\inf_{u\in H_q} J_\Ka(u)$ where 
	$H_q=\{u\in W^{s,2}(M):\, \int_Mf|u|^q=1\}$, $1\le q<2^*$. 
	Since $q<2^*$, this infimum is attained at some $u_q\in H_q$ due to Theorem \ref{subcritical}. 
	
	Using the constant test-function $1/(\int_M f dv_g)^{1/q}\in H_q$ to estimate $\mu_q$, we obtain 
	$$
	\mu_q\le J_\Ka\left(\frac{1}{(\int_M f dv_g)^{1/q}}\right) \le \dfrac{\|h\|_\infty}{(\int_M f dv_g)^{2/q}}\le C.
	$$
	Since $J_\Ka$ is coercive, we deduce that $(u_q)$ is bounded in $W^{s,2}(M)$. 
	Then $u_q\to u_0$ as $q\to 2^*$ weakly in  $W^{s,2}(M)$ and strongly in $L^2(M)$. 
	Thus $u_0$ is a weak solution of \eqref{MainEquK2}.

	An easy claim is that $\mu_q \to \mu_0$ as $q\to 2^*$. Indeed, given $\eps>0$ there is a $u\in H$ such that
	$J_\Ka(u)\leq \mu_0 + \eps$. It follows that
	$$
	\mu_q\leq J_\Ka\left( \frac{u}{(\int_M f |u|^q )^{1/q}} \right).
	$$
	It is clear that $(\int_M f|u|^q)^{1/q} \to (\int_M f|u|^{2^*})^{1/{2^*}}=1$ as $q\to 2^*$. Then
	$$
	\limsup \mu_q \leq J_\Ka(u) \leq  \mu_0 + \eps. 
	$$
	Hence, $\limsup \mu_q \leq \mu_0$ as $q\to 2^*$.
	
	Conversely, it follows from H\"older’s inequality that 
	$$
	1=\int_M f|u_q|^q dv_g = \int_M f^{1-q/2^*}\left(f^{1/2^*}|u_q| \right)^{q} dv_g \leq \left(\int_M f dv_g\right)^{1-q/2^*}\left( \int_M f|u_q|^{2^*} dv_g\right)^{q/2^*}.
	$$
	Thus we have $1\leq\liminf \left(\int_M f|u_q|^{2^*}\right)$ as $q \to 2^*$. Noting that
	$$
	\mu_0 \leq J_\Ka\left(\frac{u_q}{\left(\int_M f|u_q|^{2^*} dv_g \right)^{1/2^*}}\right)=\frac{\mu_q}{\left(\int_M f|u_q|^{2^*} dv_g \right)^{2/2^*}},
	$$
	we then get that $\mu_0 \leq \liminf\mu_q $ as $q\to 2^*$.
	It follows that $\lim\limits_{q\to 2^*} \mu_q= \mu_0$, and the above claim is proved.

	To prove that $u_0\not\equiv 0$, we use Theorem \ref{TeoBestConst}.  We write 
	\begin{eqnarray*}
		1 & = & \left(\int_M f|u_q|^q dv_g\right)^{2/q}
		\le \left(\max\,f\right)^{2/q} \left(\int_M |u_q|^{2^*}dv_g \right)^{2/2^*} |M|^{(1-q/2^*)2/q}  \\ 
		& \le & \left(\max\,f\right)^{2/q} |M|^{(1-q/2^*)2/q}  
		\left( \left(K(n, s, 2)+\eps\right)\left(\displaystyle \La_{\Ka} u_q, u_q\right)^2 + C_\eps \|u_q\|_2^2  \right)  \\ 
		& \le & \left(\max\,f\right)^{2/q} |M|^{(1-q/2^*)2/q}  
		\left( (K(n, s, 2)+\eps)2\mu_q +  (C_\eps+(K(n, s, 2)+\eps)\|h\|_\infty) \|u_q\|_2^2  \right).
	\end{eqnarray*}
	Since $\limsup_{q\to 2^*}\mu_q\le \mu_0< (2(\max\,f)^{2/2^*}K(n,s,2))^{-1} $ we deduce that 
	$\|u_0\|_2\ge C>0$ so that $u_0\not\equiv 0$.  
	
	It remains to prove that $u_0\in H$ i.e. that $\int_Mf|u_0|^{2^*}=1$. Indeed,
	from the weak convergence $u_q \rightharpoonup u_0$ on $W^{s,2}(M)$, and the fundamental property of 
	the weak limit (the norm of a weak limit is less than or equal to the infimum limit of 
	the norms of the sequence), we have
	$$
	\Vert u_0 \Vert_{s,2} \leq \liminf \ \Vert u_q \Vert_{s,2},
	$$
	from which it follows
	$$
	(\displaystyle \La_{\Ka} u_0, u_0)^2 \leq \liminf \ (\displaystyle \La_{\Ka} u_q , u_q)^2.
	$$
	As  $u_q \to u_0$ strongly on $L^{2}(M)$, we know that $\Vert u_q\Vert_2 \to \Vert u_0 \Vert_2 $ and
	$\int_M h |u_q|^2 \to \int_M h |u_0|^2$. Then it follows
	\begin{equation}\label{weakproperty}
	(\displaystyle \La_{\Ka} u_0, u_0)^2 + \int_M h |u_0|^2 \leq \liminf \left( (\displaystyle \La_{\Ka} u_q , u_q)^2 + \int_M h |u_q|^2 \right).
	\end{equation}
	Taking  $u_0$ as a test function in the equation where $u_0$ is a weak solution, we have 
	that the l.h.s. of \eqref{weakproperty} is equal to $\int_M f |u_0|^{2^*} dv_g$.  
	Analogously, the r.h.s. is $\lim_{q\to 2*} \int_M f |u_q|^{q} dv_g =1.$
	Finally we obtain
	$$
	\int_M f |u_0|^{2^*} dv_g\leq 1.
	$$
	
	To prove $ \int_M f |u_0|^{2^*} dv_g\geq1,$ we note that
$$
		\mu_0 \leq 
		\frac{ J_\Ka(u_0)}{\left(\int_M f|u_0|^{2^*} dv_g \right)^{2/2^*}}.
$$
Since
$
J_\Ka(u_0)=\frac{1}{2} \int_M f|u_0|^{2^*} dv_g,
$
we get then
$$
	\frac{1}{2}\left(\int_M f|u_0|^{2^*} dv_g \right)^{1-2/2^*}
	\geq \mu_0= \lim_{q\to 2*}\mu_q =  \lim_{q\to 2*} J_\Ka(u_q) = \lim_{q\to 2*} \frac{1}{2} \int_M f|u_q|^{q} dv_g=\frac{1}{2},
$$
and the Theorem is proved.

\qed

Using the constant test-function $v=1/\left(\int_M f dv_g\right)^{1/2^*}\in H$ we obtain 
\begin{cor}\label{ExistSolution2} If $f\ge 0$ and $h$ are smooth functions on $M$ such that
	$ J_{\Ka}$ satisfies the coercivity condition \eqref{coerc2} and
$$ 
\frac{\left(\max\,f\right)^{2/2^*}\left(\int_M h dv_g\right)}{\left(\int_M f dv_g\right)^{2/2^*}}<K(n,s,2)^{-1},
$$ 
then \eqref{Cond} holds and thus \eqref{MainEquK2} has a non-trivial solution. 
\end{cor}

\section{Test-function computations - case $p=2$}

This section provides some necessary results to prove the Theorem \eqref{TeoBestConst}. We use the assumptions \eqref{K1}-\eqref{K4} about $\Ka$ and the behaviour of the function $U$ given by \eqref{bubble} and its derivatives. 
 Unfortunately, we could not extend these results to the case $p \neq 2$ 
 because there is not enough information about the decay of the derivatives of the corresponding minimizer.

\begin{prop}\label{PropBlowup}
	Given a function $U:\R^n\to \R$ such that $(\displaystyle \La_{\Ka}U,U)<\infty$ and a point 
	$x_0\in M$ consider the function 
	$U_\eps(x)=\eps^{-\frac{n-2s}{2}}U(\frac{1}{\eps}\exp_{x_0}^{-1}(x))$ defined on $B_\delta({x_0})$ with $\delta$ smaller than the injectivity radius of $(M,g)$. 
	Then 
	\begin{equation}\label{I1} 
	\begin{split}
	& \iint_{B_\delta(x_0)\times B_\delta(x_0)} |U_\eps(x)-U_\eps(y)|^2\Ka(x,y;g)\,dv_g(x)dv_g(y)  \\ 
	& = \eps^{n+sp}\iint_{B_{\delta/\eps} \times B_{\delta/\eps} } 
	|U(x)-U(y)|^2\tilde{\Ka}(x,y;g_\eps)\,dv_{g_\eps}(x)dv_{g_\eps}(y).
	\end{split}
	\end{equation}   
	
\end{prop}

\begin{proof}

	Consider $T:B_{\delta/\eps} \to B_\delta $ given by $T(x)=\eps x$. 
	Then $T^*\exp_{x_0}^*g(x)=\eps^2 g_\eps(x)$ so that 
	$dv_{T^*(\exp_{x_0})^*g}(x) = \eps^n dv_{g_\eps}(x)$. Then we have
	\begin{eqnarray*}
	I_1 &=& \iint_{B_\delta(x_0)\times B_\delta(x_0)} |U_\eps(x)-U_\eps(y)|^2\Ka(x,y;g)\,dv_g(x)dv_g(y)  \\ 
		& = & \eps^{-n+2s}\iint_{B_\delta \times B_\delta } |U\left(\frac{x}{\eps}\right)-U\left(\frac{y}{\eps}\right)|^2
		{\Ka}(\exp_{x_0}(x),\exp_{x_0}(y);(\exp_{x_0})^*g) \,dv_{(\exp_{x_0})^*g}(x)dv_{(\exp_{x_0})^*g}(y)\\
		& = &\eps^{-n+2s} \iint_{B_{\delta/\eps} \times B_{\delta/\eps}} |U(x)-U(y)|^2{\Ka}(T(\exp_{x_0}(x)),T(\exp_{x_0}(y));T^*(\exp_{x_0})^*g)\\
		& &	\ \ \ \ \ \ \ \ \ \ \ \ \ \ \ \ \ \ \ \ \ \ \ \ \ \ \ \ \ \ \,dv_{T^*(\exp_{x_0})^*g}(x)dv_{T^*(\exp_{x_0})^*g}(y) \\
		& = & \eps^{n+2s} \iint_{B_{\delta/\eps} \times B_{\delta/\eps}} |U(x)-U(y)|^2\tilde{\Ka}( x, y; g_\eps) \,dv_{g_\eps}(x)dv_{g_\eps}(y). \\
	\end{eqnarray*}
\end{proof}


Let $\eta:[0,+\infty)\to [0,1]$ be a smooth test-function with compact support in $[0,2\delta]$ and such that $\eta\equiv 1$ in $[0,\delta]$. We choose $\delta>0$
such that $2\delta$ is smaller than the injectivity radius of $(M,g)$.  
Given a point $x_0\in M$, $\eps>0$ and a  $U\in W^{s,2}(\R^n)$ given by 
$$
U(x) = (1+|x|^2)^{-\frac{n-2s}{2}}
$$
we consider the test-function 
$u_\eps(x)=\eta(d_g(x_0,x))U_\eps(x)$ where 
$U_\eps(x)=\eps^{-\frac{n-2s}{2}}  U\left( \frac{1}{\eps} \exp_{x_0}^{-1}(x) \right)$. 

\begin{prop}
	There holds 
	\begin{equation}\label{CompGrad}
	\limsup_{\eps\to 0} \iint_{M\times M}
	|u_\eps(x)-u_\eps(y)|^2\Ka(x,y;g)\,dv_g(x)dv_g(y) 
	\le \iint_{\R^n\times \R^n} 
	\frac{|U(x)-U(y)|^2}{|x-y|^{n+2s}}\,dxdy.  
	\end{equation}
\end{prop}

The proof uses ideas of \cite[Prop. 21]{SV}. 
We split the proof into two steps. 

\begin{step1} There holds 
	\begin{equation}\label{Gradient101}
	\begin{split}
	& \iint_{M\times M} |u_\eps(x)-u_\eps(y)|^2 \Ka(x,y;g)\,dv_g(x)dv_g(y) \\ 
	& \le 
	\iint_{x\in B_{2\delta/\eps} ,\,y\in B_{2\delta/\eps} }
	|U(x)-U(y)|^2\tilde{\Ka}(x,y;g_\eps)\,dv_{g_\eps}(x)dv_{g_\eps}(y)
	+ O(\eps^{n-2s}) + O(\eps^{2s}).  
	\end{split}
	\end{equation}
\end{step1}

\begin{proof} 
	We  write 
	\begin{equation}\label{Gradient10}
	\iint_{M\times M}  |u_\eps(x)-u_\eps(y)|^2 \Ka(x,y;g)\,dv_g(x)dv_g(y) 
	= I_1+2I_2+I_3 
	\end{equation}
	where 
\begin{eqnarray*}
	I_1 &=& \iint_{B_\delta(x_0)\times B_\delta(x_0)}  |u_\eps(x)-u_\eps(y)|^2 \Ka(x,y;g)\,dv_g(x)dv_g(y) \\ 
	I_2 &=& \iint_{x\in B_\delta(x_0),\,y\not\in B_\delta(x_0)}  |u_\eps(x)-u_\eps(y)|^2 \Ka(x,y;g)\,dv_g(x)dv_g(y) \\
	I_3 &=& \iint_{x,y\not\in B_\delta(x_0)}  |u_\eps(x)-u_\eps(y)|^2 \Ka(x,y;g)\,dv_g(x)dv_g(y). \\
\end{eqnarray*}

	According to Prop. \ref{PropBlowup}, 
	\begin{equation}
	I_1 = \eps^{n+sp}\iint_{B_{\delta/\eps} \times B_{\delta/\eps} } 
	|U(x)-U(y)|^2\tilde{\Ka}(x,y;g_\eps)\,dv_{g_\eps}(x)dv_{g_\eps}(y).
	\end{equation}   
	
	We now prove 
	\begin{equation}\label{I3} 
	I_3\le C\eps^{n-2s}.
	\end{equation}
	Indeed 
	$$
	I_3\leq \Lambda \iint_{x,y\not\in B_\delta(x_0)}  \frac{|u_\eps(x)-u_\eps(y)|^2}{d_g(x,y)^{n+2s}}\,dv_g(x)dv_g(y).
	$$
	Here we use again \eqref{K3}.
Indeed for $x\in B_{2\delta}$, let
$$ \tilde u_\eps(x):=u_\eps(\exp_{x_0}(x)) = \eta(|x|)\tilde U_\eps(x), 
\qquad \tilde U_\eps(x)=\eps^{-\frac{n-2s}{2}}U(x/\eps). $$ 
Then 
$$ I_3 \leq \Lambda  \iint_{x,y\not\in B_\delta } 
\frac{|\tilde u_\eps(x)-\tilde u_\eps(y)|^2}
{d_{\exp_{x_0}^*g}(x,y)^{n+2s}}\,dv_{\exp_{x_0}^*g}(x)dv_{\exp_{x_0}^*g}(y). 
$$
There exists $C>0$ such that as bilinear forms, 
\begin{equation}\label{ControlMetric}
C^{-1}\delta_{ij}\le (\exp_{x_0}^*g)(x)\le C\delta_{ij} \qquad 
\text{for $x\in B_{2\delta} $.} 
\end{equation}
It follows that 
\begin{eqnarray*}
I_3 
& \le & C\iint_{x,y\not\in B_\delta } 
	\frac{|\tilde u_\eps(x)-\tilde u_\eps(y)|^2}{|x-y|^{n+2s}}\,dxdy \\ 
& \le & C\iint_{x \in B_{2\delta}\backslash B_\delta,y\not\in B_\delta } 
\frac{|\tilde u_\eps(x)-\tilde u_\eps(y)|^2}{|x-y|^{n+2s}}\,dxdy 
\end{eqnarray*}
where we used that $\tilde u_\eps$ is supported in $B_{2\delta}$ in the second inequality. 
Eventually, for $x,y\in\R^n\backslash B_\delta$, 
\begin{eqnarray*}
&& |\tilde u_\eps(x)-\tilde u_\eps(y)| \\ 
&& \le 1_{|x|,|y|\ge \delta,|x-y|\le \delta/2} 
\max_{\R^n\backslash B_{\delta/2}} \left(|\tilde U_\eps|+|\nabla \tilde U_\eps|\right)
|x-y| 
+ 1_{|x|,|y|\ge \delta,|x-y|\ge \delta/2} 
2 \max_{\R^n\backslash B_{\delta/2}} |\tilde U_\eps|  \\
&& \le C\eps^{n-2s}\left(1_{|x|,|y|\ge \delta,|x-y|\le \delta/2} |x-y| + 
1_{|x|,|y|\ge \delta,|x-y|\ge \delta/2}  \right).
\end{eqnarray*}
Thus 
\begin{eqnarray*}
\eps^{2s-n}I_3
&\le & C \iint_{|x|\le 2\delta,|x-y|\le \delta/2} \frac{dxdy}{|x-y|^{n+2s-2}} 
+ C \iint_{|x|\le 2\delta,|x-y|\ge \delta/2}  \frac{dxdy}{|x-y|^{n+2s}}  \\ 
&\le & 
C\iint_{|x|\le 2\delta,|t|\le \delta/2} \frac{dxdt}{|t|^{n+2s-2}} 
+ C\iint_{|x|\le 2\delta,|t|\ge \delta/2} \frac{dxdt}{|t|^{n+2s}} 
\end{eqnarray*}
which is finite. We deduce \eqref{I3}.

Concerning $I_2$ we will prove that 
\begin{equation}\label{I2}
I_2 
\le \iint_{x\in B_{\delta/\eps},\,y\in B_{2\delta/\eps} \backslash B_{\delta/\eps}} |U(x)-U(y)|^2\tilde{\Ka}(x,y;g_\eps)\,dv_{g_\eps}(x)dv_{g_\eps}(y) 
+ O(\eps^{n-2s}) + O(\eps^{2s}).
\end{equation}
To prove that we write 
$$ I_2 = I_2^1 +  I_2^2 +   I_2^3 $$ with
\begin{align*}
I_2^1 &= \iint_{x\in B_\delta(x_0),\,y\not\in B_\delta(x_0),\, d_g(x,y)<\delta/4} |u_\eps(x)-u_\eps(y)|^2\Ka(x,y;g)\,dv_g(x)dv_g(y) \\
I_2^2 &= \iint_{x\in B_\delta(x_0),\,y\in B_{2\delta}(x_0)\backslash B_\delta(x_0),\, d_g(x,y)\ge \delta/4}  |u_\eps(x)-u_\eps(y)|^2\Ka(x,y;g)\,dv_g(x)dv_g(y) \\
I_2^3 &= \iint_{x\in B_\delta(x_0),\,y\not\in B_{2\delta}(x_0),\, d_g(x,y)\ge \delta/4}|u_\eps(x)-u_\eps(y)|^2\Ka(x,y;g)\,dv_g(x)dv_g(y).\\
\end{align*}
We first check, as we did for $I_3,$ that 
\begin{equation}\label{I21} 
I_2^1\le C\eps^{n-2s}. 
\end{equation}
Indeed 
\begin{eqnarray*}
I_2^1 
&\le  & 
\Lambda \iint_{x\in B_\delta(x_0),\,y\not\in B_\delta(x_0),\, d_g(x,y)<\delta/4} \frac{|u_\eps(x)-u_\eps(y)|^2}{d_g(x,y)^{n+2s}}\,dv_g(x)dv_g(y)\\
& \le  &  
\Lambda \iint_{\delta/2\le |x|\le \delta,\,\delta\le |y|\le 3\delta/2,\, d_{\exp_{x_0}*g}(x,y)<\delta/4} 
\frac{|\tilde u_\eps(x)-\tilde u_\eps(y)|^2}{d_{\exp_{x_0}*g}(x,y)^{n+2s}}\,
dv_{\exp_{x_0}*g}(x)dv_{\exp_{x_0}*g}(y) \\ 
& \le & 
C\iint_{\delta/2\le |x|\le \delta,\,\delta\le |y|\le 3\delta/2,\, |x-y|<C\delta/4} 
\frac{|\tilde u_\eps(x)-\tilde u_\eps(y)|^2}{|x-y|^{n+2s}}\,dxdy
\end{eqnarray*}
where the constant $C$ can be chosen arbitrarily close to 1 up to taking $\delta$ small enough. 
In particular we assume that $\rho:= \delta/2-C\delta/4>0$. 
It follows that any segment $[x,y]$ remains far from $0$ since 
for any $t\in [0,1]$, $|x+t(y-x)|\ge ||x|-t|y-x||=|x|-t|y-x|\ge \rho$. 
Then $\max_{[x,y]} |\tilde U_\eps|+|\nabla \tilde U_\eps|\le C\eps^{n-2s}$. 
Thus 
\begin{eqnarray*}
\eps^{2s-n} I_2^1 
\le C \int_{|x|\le \delta}   \int_{|y|\le \delta,\, |x-y|<C\delta/4} 
\frac{1}{|x-y|^{n+2s-2}}\,dxdy \le C
\end{eqnarray*}
where we used that $s\in (0,1)$ to obtain that the inner integral is bounded by a constant. 
From there, we get \eqref{I21}.

Concerning $I_2^3$, notice that when $x\in B_\delta(x_0)$ and 
$y\not\in B_{2\delta}(x_0)$ we have $d_g(x,y)>\delta$ and $u_\eps(y)=0$. 
Thus 
$$ 
I_2^3 \le \iint_{x\in B_\delta(x_0), y\not\in B_\delta(x_0)} |u_\eps(x)|^2\Ka(x,y;g)\,dv_g(x)dv_g(y). 
$$ 
From \eqref{K3}, we have
$$ 
I_2^3 \le \Lambda \int_{x\in B_\delta(x_0), y\not\in B_\delta(x_0)} \frac{|u_\eps(x)|^2}{d_g(x,y)^{n+2s}}\,dv_g(x)dv_g(y)
\le \Lambda\delta^{-n-2s} Vol_g(M) \int_{x\in B_\delta(x_0)} |u_\eps(x)|^2\,dv_g(x). 
$$ 

In view of \eqref{BoundL2}, we obtain 
$ I_2^3 \le C_\delta \eps^{2s}. $ 

We eventually verify that 
\begin{equation}\label{I22}
I_2^2 \le 
\iint_{x\in B_\delta(x_0),\,y\in B_{2\delta}(x_0) \backslash B_\delta(x_0),\, d_g(x,y)\ge \delta/4} 
|U_\eps(x)-U_\eps(y)|^2\Ka(x,y;g)\,dv_g(x)dv_g(y)
+ O(\eps^{n-2s}) + O(\eps^{2s}). 
\end{equation}
Indeed for $x\in B_\delta(x_0)$ we have $u_\eps(x)= U_\eps(x)$ so that 
\begin{eqnarray*}
	|u_\eps(x)-u_\eps(y)|^2
	&= &|U_\eps(x)-U_\eps(y) +U_\eps(y)-u_\eps(y)|^2 \\
	& = & |U_\eps(x)-U_\eps(y)|^2 +  |U_\eps(y)-u_\eps(y)|^2 + 2 |U_\eps(x)-U_\eps(y)||U_\eps(y)-u_\eps(y)| \\
	& \le & |U_\eps(x)-U_\eps(y)|^2 +  8|U_\eps(y)|^2 + 4 |U_\eps(x)||U_\eps(y)|.
\end{eqnarray*}
As we had for $I_2^3$, we get 
$$
\iint_{x\in B_\delta(x_0),\,y\in B_{2\delta}(x_0) \backslash B_\delta(x_0),\, d_g(x,y)\ge \delta/4} |U_\eps(x)-U_\eps(y)|^2\Ka(x,y;g)\,dv_g(x)dv_g(y) 
$$
$$
\le \Lambda C_\delta  \int_{B_{2\delta}(x_0)} |U_\eps(y)|^2\,dv_g(y) 
\le  C \eps^{2s}.   
$$

Moreover for $x\in B_\delta(x_0)$ and $ y\in B_{2\delta}(x_0) \backslash B_\delta(x_0)$, 
we have $U_\eps(y)\le C\eps^{(n-2s)/2}$ so that  
$$ U_\eps(x)U_\eps(y)\le C U(\exp_{x_0}^{-1}(x)/\eps) $$ 
and then 
\begin{eqnarray*}
	& &  \iint_{x\in B_\delta(x_0),\,y\in B_{2\delta}(x_0) \backslash B_\delta(x_0),\, d_g(x,y)\ge \delta/4} 
	|U_\eps(x)||U_\eps(y)|\Ka(x,y;g)\,dv_g(x)dv_g(y)\\ 
	&&\le  \Lambda \iint_{x\in B_\delta(x_0),\,y\in B_{2\delta}(x_0) \backslash B_\delta(x_0),\, d_g(x,y)\ge \delta/4} 
	\frac{|U_\eps(x)||U_\eps(y)|}{d_g(x,y)^{n+2s}}\,dv_g(x)dv_g(y) 	\\ 
&& \le C  \iint_{|x|\le \delta,\,\delta\le |y|\le 2\delta,\, 
	d_{\exp_{x_0}^*g}(x,y)\ge \delta/4} 
\frac{U(x/\eps)}{d_{\exp_{x_0}^*g}(x,y)^{n+2s}}\,dv_{\exp_{x_0}^*g}(x)dv_{\exp_{x_0}^*g}(y) \\
&& \le C\eps^{-n-2s} \iint_{|x|\le \delta,\,\delta\le |y|\le 2\delta,\, 
	d_{\exp_{x_0}^*g}(x,y)\ge \delta/4} 
\frac{U(x/\eps)}{d_{g_\eps}(x,y)^{n+2s}}\,dv_{\exp_{x_0}^*g}(x)dv_{\exp_{x_0}^*g}(y) \\ 
&& \le C\eps^{n-2s} \iint_{|x|\le \delta/\eps,\,\delta/\eps\le |y|\le 2\delta/\eps,\, 
	|x-y|\ge \delta/(2/\eps)} \frac{U(x)}{|x-y|^{n+2s}}\,dxdy \\ 
&& \le C\eps^{n-2s} \int_{|x|\le \delta/\eps} U(x)\,dx 
\int_{|\xi|\ge \delta/(2/\eps)} |\xi|^{-n-2s}\,d\xi.  
\end{eqnarray*}
Since $\int_{|x|\le \delta/\eps} U(x)\,dx \le C\eps^{-2s}$ and 
$\int_{|\xi|\ge \delta/(2/\eps)} |\xi|^{-n-2s}\,d\xi \le C\eps^{2s}$ we obtain 
$$ \iint_{x\in B_\delta(x_0),\,y\in B_{2\delta}(x_0) \backslash B_\delta(x_0),\, d_g(x,y)\ge \delta/4} 
|U_\eps(x)||U_\eps(y)|\Ka(x,y;g)\,dv_g(x)dv_g(y) \le C\eps^{n-2s}. $$ 
We deduce \eqref{I22}. 

\end{proof}

\begin{step2}
\eqref{CompGrad} holds. 
\end{step2}
	\begin{proof}
	In view of the previous Step, it is enough to prove that 
	\begin{equation}\label{limsupB2delta}
	\limsup_{\eps\to 0}
	\eps^{n+sp}\iint_{B_{2\delta/\eps} \times B_{2\delta/\eps}}
	|U(x)-U(y)|^2\tilde\Ka(x,y;g_\eps)\,dv_{g_\eps}(x)dv_{g_\eps}(y) 
	\le 
	\iint_{\R^n\times\R^n } \frac{|U(x)-U(y)|^2}{|x-y|^{n+2s}}\,dxdy.
	\end{equation}  
	
	Given $R>0$ let 
	$$ \eps_R = \iint_{(\R^n\times \R^n)\backslash (B_R \times B_R )} 
	\frac{|U(x)-U(y)|^2}{|x-y|^{n+2s}}\,dxdy $$ 
	which goes to 0 as $R\to +\infty$  by dominated convergence since $U\in W^{s,2}(\R^n)$. 
	For $\eps>0$, small enough so that $R<2\delta/\eps$, we split the integral in the l.h.s of \eqref{limsupB2delta} as 
	\begin{eqnarray}\label{Gradient30}
	\iint_{B_R \times B_R }	|U(x)-U(y)|^2\tilde \Ka(x,y;g_\eps)\,dv_{g_\eps}(x)dv_{g_\eps}(y)  \hspace{0.5cm} +  \\
	\iint_{(B_{2\delta/\eps}\times B_{2\delta/\eps})\backslash (B_R \times B_R )}	|U(x)-U(y)|^2\tilde 
	\Ka(x,y;g_\eps)\,dv_{g_\eps}(x)dv_{g_\eps}(y). 
	\end{eqnarray}
	Then, for the second integral it holds
	\begin{eqnarray*}
		&& \iint_{(B_{2\delta/\eps}\times B_{2\delta/\eps})\backslash (B_R \times B_R )} 
		|U(x)-U(y)|^2\tilde \Ka(x,y;g_\eps)\,dv_{g_\eps}(x)dv_{g_\eps}(y)  \\ 
		&& \le \Lambda \iint_{(B_{2\delta/\eps}\times B_{2\delta/\eps})\backslash (B_R \times B_R )} 
		\frac{|U(x)-U(y)|^2}{d_{g_\eps}(x,y)^{n+2s}}\,dv_{g_\eps}(x)dv_{g_\eps}(y)  \\ 
		&& \le C \iint_{(B_{2\delta/\eps}\times B_{2\delta/\eps})\backslash (B_R \times B_R )} 
		\frac{|U(x)-U(y)|^2}{|x-y|^{n+2s}}\,dxdy \  \le C \eps_R.  \\ 
	\end{eqnarray*}
	Thus 
	$$ \iint_{B_{2\delta/\eps} \times B_{2\delta/\eps}}
	|U(x)-U(y)|^2\tilde\Ka(x,y;g_\eps)\,dv_{g_\eps}(x)dv_{g_\eps}(y)  
	$$
	$$
	=  \iint_{B_R \times B_R } 	|U(x)-U(y)|^2\tilde\Ka(x,y;g_\eps)\,dv_{g_\eps}(x)dv_{g_\eps}(y) 
	+ O(\eps_R). $$ 
	Moreover, for a given $R>0$, we can send $\eps\to 0$ in \eqref{Gradient30}. From \eqref{K4}, we obtain 
	\begin{eqnarray*}
		&& \lim_{\eps\to 0}	
		\eps^{n+sp}\iint_{B_R \times B_R } 
		|U(x)-U(y)|^2\tilde\Ka(x,y;g_\eps)\,dv_{g_\eps}(x)dv_{g_\eps}(y)   \\
		&& = \iint_{B_R \times B_R } \frac{|U(x)-U(y)|^2}{|x-y|^{n+2s}}\,dxdy  = \iint_{\R^n\times \R^n} \frac{|U(x)-U(y)|^2}{|x-y|^{n+2s}}\,dxdy - \eps_R.  
	\end{eqnarray*}
	
\end{proof}

\begin{prop}\label{LpNorm}
There hold 
\begin{equation}\label{BoundL2}
 \int_M |u_\eps|^2\,dv_g \le C\eps^{2s} 
\end{equation}
and 
\begin{equation}\label{BoundL2*}
 \lim_{\eps\to 0} \int_M |u_\eps|^{2^*}\,dv_g = \int_{\R^n} |U|^{2^*}\,dx. 
\end{equation}
\end{prop}

\begin{proof} 
We have 
$$
\int_M |u_\eps|^2\,dv_g
= \eps^{-n+2s}\int_{B_{2\delta}(x_0)} |U_\eps(x)|^2\,dv_{g}
\le \eps^{2s}\int_{B_{2\delta/\eps}} |U|^2\,dv_{g_\eps}
\le C\eps^{2s}\int_{\R^n} |U|^2\,dx.
$$
Moreover, given $R>0$, we have for $\eps$ small enough so that $R\eps<\delta$, that 
$$
\int_M |u_\eps|^{2^*}\,dv_g 
 = \eps^{-n}  \int_{B_{R\eps}(x_0)} |U_\eps|^{2^*}\,dv_g
 + \eps^{-n}  \int_{B_{2\delta}(x_0)\backslash B_{R{\eps}}(x_0)} |u_\eps|^{2^*}\,dv_g  
= I + J.
$$
We have that 
$ I$ tends to  $\int_{\R^n} |U|^{2^*}\,dx$ as ${R\to +\infty}$ and ${\eps\to 0}$. Furthermore, 
$$
 J\le \int_{B_{2\delta/\eps} \backslash B_R } |U|^{2^*}\,dv_{g_\eps}  
\le C\int_{B_{2\delta/\eps} \backslash B_R } |U|^{2^*}\,dx
$$ 
and thus $ J$ goes to $0$ as ${R\to +\infty}$ and ${\eps\to 0}$. 
\end{proof}

\bibliographystyle{plain}
\bibliography{fract.bib}

\begin{thebibliography}{10}

\bibitem{adams1999function}
David~R Adams and Lars~I Hedberg.
\newblock {\em Function spaces and potential theory}, volume 314.
\newblock Springer Science \& Business Media, 1999.

\bibitem{alonso2018integral}
Diego Alonso-Oran, Antonio Cordoba, and Angel~D Martinez.
\newblock Integral representation for fractional laplace--beltrami operators.
\newblock {\em Advances in Mathematics}, 328:436--445, 2018.

\bibitem{BBM}
Jean Bourgain, Haim Brezis, and Petru Mironescu.
\newblock Another look at sobolev spaces, in:.
\newblock {\em Optimal Control and Partial Differential Equations. A Volume in
  Honour of A. Bensoussan’s 60th Birthday,}.

\bibitem{BBM2002}
Jean Bourgain, Ha{\"\i}m Brezis, and Petru Mironescu.
\newblock Limiting embedding theorems for $w^{s,p}$ when $s\nearrow1$ and
  applications.
\newblock {\em Journal d'Analyse Math{\'e}matique}, 87(1):77--101, 2002.

\bibitem{brezis1983positive}
H~Brezis and L~Nirenberg.
\newblock Positive solutions of nonlinear elliptic equations involving critical
  sobolev exponents.
\newblock {\em SMR}, 398:2, 1983.

\bibitem{TPV}
Tommaso Bruno, Marco~M Peloso, and Maria Vallarino.
\newblock Besov and triebel--lizorkin spaces on lie groups.
\newblock {\em Mathematische Annalen}, pages 1--43, 2019.

\bibitem{caffa}
Luis Caffarelli.
\newblock Non-local diffusions, drifts and games.
\newblock In {\em Nonlinear partial differential equations}, pages 37--52.
  Springer, 2012.

\bibitem{caffarelli2012non}
Luis Caffarelli.
\newblock Non-local diffusions, drifts and games.
\newblock In {\em Nonlinear partial differential equations}, pages 37--52.
  Springer, 2012.

\bibitem{chang2011fractional}
Sun-Yung~Alice Chang and Maria del Mar~Gonzalez.
\newblock Fractional laplacian in conformal geometry.
\newblock {\em Advances in Mathematics}, 226(2):1410--1432, 2011.

\bibitem{chen2006classification}
Wenxiong Chen, Congming Li, and Biao Ou.
\newblock Classification of solutions for an integral equation.
\newblock {\em Communications on pure and applied mathematics}, 59(3):330--343,
  2006.

\bibitem{dNS}
Pablo~Luis De~N{\'a}poli and Pablo~Ra{\'u}l Stinga.
\newblock Fractional laplacians on the sphere, the minakshisundaram zeta
  function and semigroups.
\newblock {\em arXiv preprint arXiv:1709.00448}, 2017.

\bibitem{NPV}
Eleonora Di~Nezza, Giampiero Palatucci, and Enrico Valdinoci.
\newblock Hitchhiker's guide to the fractional sobolev spaces.
\newblock {\em Bulletin des sciences math{\'e}matiques}, 136(5):521--573, 2012.

\bibitem{djadli2000paneitz}
Zindine Djadli, Emmanuel Hebey, and Michel Ledoux.
\newblock Paneitz-type operators and applications.
\newblock {\em Duke Mathematical Journal}, 104(1):129--169, 2000.

\bibitem{frank2008non}
Rupert~L Frank and Robert Seiringer.
\newblock Non-linear ground state representations and sharp hardy inequalities.
\newblock {\em Journal of Functional Analysis}, 255(12):3407--3430, 2008.

\bibitem{gilboa2009nonlocal}
Guy Gilboa and Stanley Osher.
\newblock Nonlocal operators with applications to image processing.
\newblock {\em Multiscale Modeling \& Simulation}, 7(3):1005--1028, 2009.

\bibitem{gonzalez2013fractional}
Mar{\'\i}a del~Mar Gonz{\'a}lez~Nogueras and Jie Qing.
\newblock Fractional conformal laplacians and fractional yamabe problems.
\newblock {\em Analysis \& PDE}, 6(7):1535--1576, 2013.

\bibitem{GZZ}
Lifeng Guo, Binlin Zhang, and Yadong Zhang.
\newblock Fractional p-laplacian equations on riemannian manifolds.
\newblock {\em Electronic Journal of Differential Equations}, 2018(156):1--17,
  2018.

\bibitem{hebey}
Emmanuel Hebey.
\newblock Variational methods and elliptic equations in riemannian geometry
  workshop on recent trends in nonlinear variational problems notes from
  lectures at ictp.
\newblock 2003.

\bibitem{kim2017existence}
Seunghyeok Kim, Monica Musso, and Juncheng Wei.
\newblock Existence theorems of the fractional yamabe problem.
\newblock {\em Analysis \& PDE}, 11(1):75--113, 2017.

\bibitem{KKP}
Janne Korvenp{\"a}{\"a}, Tuomo Kuusi, and Giampiero Palatucci.
\newblock The obstacle problem for nonlinear integro-differential operators.
\newblock {\em Calculus of Variations and Partial Differential Equations},
  55(3):63, 2016.

\bibitem{lee2018introduction}
John~M Lee.
\newblock {\em Introduction to Riemannian manifolds}, volume 176.
\newblock Springer, 2018.

\bibitem{P}
Giampiero Palatucci.
\newblock The dirichlet problem for the p-fractional laplace equation.
\newblock {\em Nonlinear Analysis}, 177:699--732, 2018.

\bibitem{servadei2013yamabe}
Raffaella Servadei.
\newblock The yamabe equation in a non-local setting.
\newblock {\em Advances in Nonlinear Analysis}, 2(3):235--270, 2013.

\bibitem{SV2013}
Raffaella Servadei and Enrico Valdinoci.
\newblock Variational methods for non-local operators of elliptic type.
\newblock {\em Discrete \& Continuous Dynamical Systems-A}, 33(5):2105, 2013.

\bibitem{SV}
Raffaella Servadei and Enrico Valdinoci.
\newblock The brezis-nirenberg result for the fractional laplacian.
\newblock {\em Transactions of the American Mathematical Society},
  367(1):67--102, 2015.

\bibitem{silvestre}
Luis Silvestre.
\newblock Regularity of the obstacle problem for a fractional power of the
  laplace operator.
\newblock {\em Communications on Pure and Applied Mathematics: A Journal Issued
  by the Courant Institute of Mathematical Sciences}, 60(1):67--112, 2007.

\bibitem{triebeltheoryof}
Hans Triebel.
\newblock Theory of function spaces.
\newblock Birkhàuser Verlag, Basel, 1983.

\bibitem{T2}
Hans Triebel.
\newblock Spaces of besov-hardy-sobolev type on complete riemannian manifolds.
\newblock {\em Arkiv f{\"o}r Matematik}, 24(1-2):299--337, 1985.

\bibitem{T}
Hans Triebel.
\newblock Characterizations of function spaces on a complete riemannian
  manifold with bounded geometry.
\newblock {\em Mathematische Nachrichten}, 130(1):321--346, 1987.

\end{thebibliography}

\end{document}